\documentclass[a4paper,11pt]{article}
\usepackage{amsmath,amsthm,amsfonts,amssymb,bm} % 数学宏包
\usepackage{times}                       % 使用 Times New Roman 字体
\usepackage{CJK,CJKnumb,CJKulem}         % 中文支持宏包
\usepackage[left=2.5cm,right=2.5cm,top=2.5cm,bottom=2.5cm]{geometry}
\usepackage{color}
\usepackage{appendix}
\usepackage{diagbox}
\usepackage{multirow}
\usepackage{booktabs}
\usepackage{threeparttable}             % 支持彩色
\usepackage{caption}
\usepackage{chngcntr}
\usepackage{hyperref}
\usepackage{authblk}
\usepackage{mathrsfs}
\usepackage{qcircuit}
\usepackage{cases}

\counterwithin{figure}{section}
\counterwithin{table}{section}

\usepackage{graphicx,psfrag}                    % 图形宏包
\usepackage{subfigure}
\numberwithin{equation}{section}
\usepackage[capitalise]{cleveref}

\newtheorem{theorem}{Theorem}[section] %%定理
\newtheorem{lemma}{Lemma}[section]     %%% 引理
\newtheorem{remark}{Remark}[section]
\newtheorem{example}{Example}[section]

\begin{document} % 开始正文
\date{}
\title{A decoupled and structure-preserving  direct discontinuous Galerkin  method for the Keller--Segel  Model}
\author[a]{Xu Yin}
\author[a]{Xinyi Lan}
\author[a,*]{Yuzhe Qin}
\affil[a]{School of Mathematics and Statistics, Shanxi University, Taiyuan, 030006, China}
\affil[*]{Corresponding author, yzqin@sxu.edu.cn}
\renewcommand*{\Affilfont}{\small\it} % 修改机构名称的字体与大小
\renewcommand\Authands{ and } % 去掉 and 前的逗号
\maketitle
\begin{abstract}
In this work, we develop a novel numerical scheme to solve the classical Keller--Segel (KS) model 
which simultaneously preserves its intrinsic mathematical structure and achieves optimal accuracy. 
The model is reformulated into a gradient flow structure using the energy variational method, 
which reveals the inherent energy dissipative dynamics of the system. 
Based on this reformulation, 
we construct a structure-preserving discretization by semi-implicit method in time 
and the direct discontinuous Galerkin (DDG) method in space, 
resulting in a stable and high-order accurate approximation. 
The proposed scheme enjoys several desirable properties:
(i) energy stability, ensuring discrete free energy dissipation;
(ii) exact conservation of mass for the cell density;
(iii) positivity preservation of the numerical cell density, enforced via a carefully designed limiter; and
(iv) optimal convergence rate, with first-order accuracy in time and $(k+1)$-th order accuracy 
in space for polynomials of degree $k$. 
We provide rigorous theoretical analysis that substantiate these properties. 
In addition, extensive numerical experiments, including benchmark problems exhibiting pattern formation 
and near blow-up behavior, are conducted to validate the theoretical results and demonstrate the robustness, 
efficiency, and accuracy of the proposed method. The approach offers a flexible and reliable framework 
for structure-preserving numerical simulation of chemotaxis models and other gradient flow-type systems.
\end{abstract}

\section{Introduction} 

The Keller--Segel (KS) equations, 
firstly proposed by Patlak \cite{Pat1953} and Keller \& Segel \cite{Kel1970,Kel1971}, 
are originally formulated to model chemotaxis-driven cellular aggregation in biological systems. 
This coupled system of parabolic equations describes the interaction between cell density $u$ and chemical attractant concentration $c$, 
exhibiting rich dynamics such as pattern formation and finite-time blow-up phenomenon\cite{che2018,chj1981}.
The solutions to the KS equations are widely acknowledged for their inherent properties of mass conservation, positivity of concentration or density, and energy dissipation.

   A plethora of analytical results regarding the solutions of the classical KS model have been thoroughly investigated by researchers. These results encompass a wide range of aspects, such as existence, uniqueness, boundedness, and blow-up phenomena, as documented in references \cite{gajewski1998,corrias2014}.
For example,   Corrias et al. \cite{corrias2004} analyzed the global existence of weak solutions to the parabolic-elliptic KS equations and demonstrated that when the initial value is large, the solution will undergo blow-up or Dirac aggregation. Subsequently, Nagai et al. \cite{nagai2018} investigated the boundedness of non-negative solutions to this type of equation. Osaki et al. \cite{osaki2001} studied the existence and boundedness of global solutions to the one-dimensional parabolic-parabolic KS equations under the assumption of smooth initial values.  In addition, Winkler has done a great deal of work on the properties of solutions to this type of model \cite{winkler2010,winkler20101,winkler2013}. Generally, it is difficult for us to obtain the exact solutions to the KS equations. Therefore, efficient numerical simulation techniques are of great significance for the study of the KS equations. Ideally, numerical methods should preserve these key structural properties while maintaining accuracy and robustness across different dynamical regimes, 
including smooth aggregation and singular blow-up. 
   
Over the past few decades, significant research efforts have been directed toward developing efficient, positivity-preserving, and energy stable numerical schemes for the KS equations.
Saito et al. employed upwind finite difference schemes in space and semi-implicit Euler schemes in time, ensuring the conservation property in the discrete $L^1$ norm \cite{saito2005}. 
Based on this work, finite difference positivity-preserving and mass-conserving methods were further developed \cite{saito2007}.
Regarding the finite volume method (FVM), Filbet solved the KS model, analyzing the existence, uniqueness, and convergence of the numerical solutions \cite{ fil2006}. 
Chertock et al. designed an upwind-based finite volume method that maintained second-order accuracy while preserving the positivity of the numerical solutions \cite{che2008}. 
Subsequently, Zhou et al. derived error estimates for the finite volume method \cite{zhou2017}, allowing a first-order numerical scheme to maintain mass conservation and positivity without the restriction of the CFL condition. 
As for the finite element method (FEM),  Norikazu et al. constructed a conservative upwind finite element method for the KS equations, which preserved the positivity and mass conservation of the solutions and provided error estimates \cite{saito2007}. 
Sulman et al. proposed an efficient adaptive moving mesh finite element method for the KS equations, proving its positivity-preserving property and validating its effectiveness through numerical experiments \cite{sul2019}.  %Within this scheme, all the aforesaid properties are met at the discrete level.
% %	However, the temporal precision of this scheme is merely ﬁrst-order, and the spatial precision is second-order. 
Discontinuous Galerkin (DG) methods, renowned for their flexibility in handling complex geometries and achieving high-order accuracy, have shown promise in solving nonlinear parabolic systems. 
More general information about DG methods for elliptic, parabolic, and hyperbolic PDEs can be found in recent books  (see e.g., \cite{hes2007,ri2008}).
However, conventional DG formulations for the KS system typically rely on stabilization techniques or post-processing to enforce structural properties, which can introduce artificial dissipation or complicate convergence analysis. 
These limitations underscore the need for a \textit{direct} and \textit{structure-preserving} DG framework that inherently respects the mathematical and physical constraints of the KS model. 

In this work, we propose a DDG scheme for the KS system, designed to unify structure preservation, high-order accuracy, and rigorous convergence analysis in a broader setting.
The DG methods have been successfully applied to a wide variety of problems ranging from
the solid mechanics to the fluid mechanics \cite{Aizinger2000,cockburn1998,Girault2005,sun2005}.
The DDG method is a special class of the DG methods introduced in \cite{liu2009,liu2010} speciﬁcally for diffusion.
It has been successfully applied to various application problems, including linear and nonlinear Poisson equations \cite{huang2012, yin2014,cao2025} and Fokker--Planck type equations\cite{liu2016, liu2017}. 
Unlike traditional DG methods that rely on local reconstructions for flux terms, 
the DDG approach directly computes numerical fluxes by enforcing continuity conditions on the solution’s derivatives across element interfaces. 
This feature is particularly advantageous for handling the second-order diffusion terms in the KS system, as it avoids the need for auxiliary variables or mixed formulations. 
The DDG method’s inherent flexibility in accommodating irregular meshes and its ability to preserve local conservation laws make it well-suited for resolving sharp gradients and blow-up structures inherent to chemotactic collapse. 
To address the stiffness introduced by diffusion and chemotactic drift terms, we employ a semi-implicit time-stepping scheme. 
The implicit treatment decouples stability constraints from spatial resolution, enabling the use of larger time steps without sacrificing accuracy. 
The nonlinear system arising from the implicit discretization is solved using a Newton-Krylov solver, augmented with preconditioning techniques to mitigate the ill-conditioning caused by strong chemotactic coupling.
	
The remainder of this paper is organized as follows. Section \ref{sec:eqns} introduces the KS model and its mathematical properties. Section \ref{sec:scheme} details the DDG spatial discretization and semi-implicit time-stepping scheme. In Section \ref{sec:property}, the mass conservation and positivity properties of the numerical solution are studied. Optimal error estimates are established
in Section \ref{sec:error}. Numerical experiments in Section \ref{sec:results} validate the method’s accuracy, efficiency, and robustness in capturing blow-up phenomena. Conclusions and future directions are discussed in Section \ref{sec:conclusions}.

Throughout this paper, we adopt standard notations for Sobolev spaces.  For instance, we have $W^{m,p} (D)$ on a sub-domain $D \subset \Omega$, which is equipped with the norm $\|\cdot\|_{m,p,D}$ and the semi-norm $| \cdot |_{m,p,D}$. When $D =\Omega$, we omit the index $D$. In the case where $p = 2$, we define $W^{m,p} (D) = H^m(D)$, $\|\cdot\|_{m,p,D}=\|\cdot\|_{m,D}$,  and $| \cdot |_{m,p,D} = | \cdot |_{m,D}$.  The notation $A \lesssim B$ implies that $A$ can be bounded by $B$ multiplied by a positive constant. 
%   %  1. model background 

%    % 2. DDG method 

%    % 3. time discretization 

%    % 4. numerical challenge 

%    % 5. our contribution 

%    % 6. outline of this paper 
    
\section{Keller--Segel system and its energy law}\label{sec:eqns}
The KS system is composed of a set of parabolic partial differential equations shown as follows, 
\begin{subequations}\label{eqn:ori_KS}
\begin{numcases}{}
u_t=\nabla \cdot (D \nabla u-\chi \varphi(u) \nabla c),\quad \mbox{in} \quad  [0,T]\times \Omega\label{eqn:ori_u}\\
\beta c_t=\Delta c-\alpha c+u \quad \mbox{in} \quad [0,T]\times \Omega\label{eqn:ori_c}\\
u(0,x)=u_0,\quad c(0,x)=c_0,\quad u_0\in(0,1), \quad c_0>0,\label{eqn:ori_ini}
\end{numcases}
\end{subequations}
where $u$ and $c$ are the cell density function and chemical attractant concentration, respectively. 
The time derivatives of $u$ and $c$ are defined by $u_{t}$ and $c_{t}$. 
Cells can move randomly by diffusion with the diffusion coefficient given by $D$,
and $\chi$ is a coefficient used to represent the strength of chemotaxis. 
$\beta > 0$ is a small parameter describing the time scale. 
For simplicity, the periodic boundary conditions are set to make the system self-closed.
To make the numerical scheme more convenient for analysis, we rewrite the above original KS model \eqref{eqn:ori_KS} into the following equivalent form: 
\begin{subequations}\label{eqn:KS}
\begin{numcases}{}
u_t =\nabla \cdot (\chi \varphi(u) \nabla \frac{\delta E}{\delta u}),\label{eqn:u}\\
\beta c_t=-\frac{\delta E}{\delta c},\label{eqn:c}\\
\frac{\delta E}{\delta u}=Bg(u)-c,\label{eqn:dEdu}\\
\frac{\delta E}{\delta c}=-\Delta c+\alpha c-u,\label{eqn:dEdc}
\end{numcases}
\end{subequations}
where $u$ and $c$ have the same definition as in \eqref{eqn:ori_KS}. 
Here we omit the initial conditions for simplicity. 
Additionally, we have $\beta>0$, $B=\frac{D}{\chi}$, $F(u)=u\ln u+(1-u)\ln (1-u)$, $g(u)=F'(u)$ and $\varphi(u)=\frac{1}{g'(u)}=u(1-u)$.

The KS system \eqref{eqn:KS} performs a gradient flow structure
\begin{equation*}
\frac{\mathrm{d}E}{\mathrm{d}t}=-\int_{\Omega}\left(\chi \varphi\left(\nabla \frac{\delta E}{\delta u}\right)^2+\frac{1}{\beta} \left(\frac{\delta E}{\delta c}\right)^2\right) \mathrm{d}\bm{x} \leq 0,
\end{equation*}
with the following free energy
\begin{equation*}
E[u,c]=\int_{\Omega} B\left( u\ln u+(1-u)\ln(1-u) \right) -uc+\frac{1}{2} (|\nabla c|^2+\alpha c^2) \mathrm{d}\bm{x},
\end{equation*}
under the periodic boundary condition, 
where $\delta$ denotes the variational derivative.
For this PDE system, a positivity-preserving property, that is, $0 < u < 1$, can be
theoretically justified due to the logarithmic terms appearing in $u$.
\begin{lemma}
Suppose that $\Omega \subset \mathbb{R}^{d}(d=1,2,3)$. 
For simplicity, we assume $\Omega= (0, L_x) \times (0, L_y)$ with $L_x = L_y := L > 0$.
Let the function  $u, c: \Omega \rightarrow \mathbb{R}$ are periodic and sufficiently regular. Define the following energy functions
\begin{subequations}
\renewcommand{\theequation}{}% 清空主编号
\begin{equation*}
\left\{
\begin{aligned}
E_c[u,c]&=\int_{\Omega}B F(u)+\frac{1}{2} |\nabla c|^2+\frac{1}{2}\alpha c^2+\frac{1}{2}\Bigl(\frac{1}{\gamma}u-\gamma c\Bigr)^2 \,\mathrm{d}\bm{x},\\
E_e[u,c]&=\int_{\Omega}  \frac{1}{2\gamma^2}u^2+ \frac{1}{2}\gamma^2 c^2 \,\mathrm{d}\bm{x},
\end{aligned}
\right.
\end{equation*}
\end{subequations}
where $F(u)=u\ln u+(1-u)\ln (1-u)$, $\gamma$  is a positive number. Then $E_c[u,c]$ and $E_e[u,c]$ are both convex with respect to $u$ and $c$, with $	E[u,c]=	E_c[u,c]-	E_e[u,c]$.
\end{lemma}
\begin{proof}
We focus on the convexity analysis of $E_c[u,c]$ and $E_e[u,c]$. Denote 
\begin{eqnarray*}
&    e_c({\bm v})\triangleq  e_c(u,c,c_x,c_y)=B F(u)+\frac{1}{2} |\nabla c|^2+\frac{1}{2}\alpha c^2+\frac{1}{2}(\frac{1}{\gamma}u-\gamma c)^2\\
   &e_e(\bm v)\triangleq e_e(u,c,c_x,c_y)=\frac{1}{2\gamma^2}u^2+ \frac{1}{2}\gamma^2 c^2,
\end{eqnarray*}
where
\begin{equation*}
    {\bm v}=(v_1,v_2,v_3,v_4)\triangleq (u,c,c_x,c_y).
\end{equation*}
    Then we have 
    \begin{equation*}
        E_c[u,c]=\int_{\Omega} e_c({\bm v}) \mathrm{d}{\bm x}, \quad 
        E_e[u,c]=\int_{\Omega} e_e({\bm v}) \mathrm{d}{\bm x},
    \end{equation*}
    and the following inequalities are derived,
    \begin{eqnarray*}
        &\partial_{v_1}^2 e_c(v_1,v_2,v_3,v_4)=B\frac
        {1}{u}+B\frac
        {1}{1-u}+\frac
        {1}{\gamma^2},\\
   &  \partial_{v_2}^2 e_c(v_1,v_2,v_3,v_4)=\alpha+\gamma^2,\\
& \partial_{v_3}^2    e_c(v_1,v_2,v_3,v_4)=\partial_{v_4}^2    e_c(v_1,v_2,v_3,v_4)=1,\\
& \partial_{v_1 v_2}^2    e_c(v_1,v_2,v_3,v_4)=-1,\\
&  \partial_{v_1}^2 e_e(v_1,v_2,v_3,v_4)=\frac{1}{\gamma^2},\\
&  \partial_{v_2}^2 e_e(v_1,v_2,v_3,v_4)={\gamma^2}.
    \end{eqnarray*}
    For $e_c$, all leading principal minors of the Hessian are positive for   
 $0<u<1$ and $B,\alpha, \gamma$ be  positive numbers. And for $e_e$ the Hessian is the diagonal matrix
 $diag(\frac{1}{\gamma^2}, \gamma^2,0,0)$ which is positive semi-definite.  Consequently both 
  $e_c(\bm v)$ and $e_e(\bm v)$ are convex. Therefore, we obtain the following inequality, according to the definition of convex function,
\begin{eqnarray*}
    e_c(\lambda {\bm w}+(1-\lambda){\bm v})\leq 
    \lambda e_c(\bm w)+(1-\lambda)e_c(\bm v),\\
    e_e(\lambda {\bm w}+(1-\lambda){\bm v})\leq 
    \lambda e_e(\bm w)+(1-\lambda)e_e(\bm v),
\end{eqnarray*}
where $\lambda\in(0,1)$. Integrating both sides of above inequalities lead to 
\begin{eqnarray*}
    E_c[\lambda u_1+(1-\lambda)u_2,\lambda c_1+(1-\lambda)c_2]\leq 
     \lambda E_c[u_1,c_1]+(1-\lambda)E_c[u_2,c_2],\\
     E_e[\lambda u_1+(1-\lambda)u_2,\lambda c_1+(1-\lambda)c_2]\leq 
     \lambda E_e[u_1,c_1]+(1-\lambda)E_e[u_2,c_2].
\end{eqnarray*}
Then both $E_c[u,c]$ and $E_e[u,c]$ are convex with respect to $u$ and $c$. 
\end{proof}
\section{Numerical scheme}\label{sec:scheme}

In this section, our primary objective is to develop a numerical scheme that preserves the essential physical and mathematical properties of the underlying system. 
Specifically, the goal is to ensure that the scheme maintains the total energy dissipation, which is crucial for accurately modeling dissipative processes. %Additionally, we focus on conserving the mass of the cell density function, a fundamental requirement for many biological and physical systems. Furthermore, we strive to preserve the positivity of the solution, which is essential for maintaining the physical meaning and stability of the numerical method.
  
\subsection{Notations and auxiliary results}
For simplicity, let $T_h$ be a quasi-uniform rectangular partition of the domain $\Omega$ with mesh size
$$
h = \max_{\tau\in T_h}\operatorname{diam}(\tau).
$$
Denote by $E_h$ the set of all edges, by $E_h^0\subset E_h$ the set of interior edges, and by $E_h^B$ the set of boundary edges.
For every edge $e\in E_h$, let $h_e = \operatorname{diam}(e)$.
For any interior edge $e\in E_h^0$, let $\tau_1$ and $\tau_2$ be the two rectangles sharing $e$, and orient the unit normal $\bm{n}$ from $\tau_1$ to $\tau_2$.
For a scalar function $w$ we define on $e$
\begin{equation*}
\{w\}=\frac{1}{2}(w|_{\tau_1}+w|_{\tau_2}), \quad [w]=w|_{\tau_2}-w|_{\tau_1}, \quad \mbox{on} \quad  e\in \partial \tau_1\cap \partial \tau_2.
\end{equation*}

Denote by $V_h$ the discrete finite element space associated with $T_h$,
and $L^2_{per}$ the function space in $L_2$ space satisfying the periodic boundary condition.
Define
\begin{equation*}
V_h=\{ v\in L^2(\Omega): v|_\tau\in \mathbb{P}_k(x)\times \mathbb{P}_k(y), \forall \tau\in T_h\}, \quad V_h^0=V_h \cap L^2_{per}, 
\end{equation*}
where $\mathbb{P}_k$ denotes the space of polynomial functions of degree at most $k$.

\subsection{Spatial discretization}
We consider KS system \eqref{eqn:KS} with periodic boundary condition.
The semi-discrete DDG scheme is to find $u_h, c_h\in V_h^0$, such that for all $v, \theta \in V_h^0$,
\begin{subequations}\label{semi}
\begin{align}
\int_{\tau }\partial_t u_h v \mathrm{d}\tau =
& -\int_{\tau } \chi \varphi(u_h) \nabla (B g(u_h)-c_h) \nabla v \mathrm{d}\tau \nonumber\\
& +\int_{\partial \tau }\chi \varphi(u_h)( \widehat{\partial_n (B g(u_h)}-\widehat{c}_h) v \nonumber\\
& +((B g(u_h)-c_h)-\{ B g(u_h)-c_h\})\partial_n v)\mathrm{d}s,\\
\beta \int_{\tau} \partial_t c_h \theta \mathrm{d}\tau =
& -\int_{\tau}\nabla c_h \nabla \theta \mathrm{d}\tau
+\int_{\partial \tau } (\widehat{\partial_n c_h}\theta +(c_h-\{c_h\})\partial_n \theta)\mathrm{d}s \nonumber\\
& -\int_{\tau }\alpha c_h \theta \mathrm{d}\tau +\int_{\tau} u_h \theta \mathrm{d}\tau.
\end{align}
\end{subequations}
The term $\widehat{\partial_n w}$ denotes the numerical fluxes, which is defined at the cell interfaces $e$ by
\begin{equation*}
\widehat{\partial_n w}=\beta_0 \frac{[w]}{h_e}+\{\partial_n w\}+\beta_1 h_e[\partial_n^2 w].
\end{equation*}
There exists a large group of admissible coefficient pairs $(\beta_0,\beta_1)$ that ensure stability or satisfy certain positivity-preserving properties of DDG method. (see e.g. \cite{liu2010}).
Noticing that the interface-corrected DDG scheme \eqref {semi} employed here was introduced in \cite{liu2010} as an enhancement of the DG method presented in \cite{liu2009} for diffusion problems.
		
%Due to the nonlinear terms that appeared in the first equation in \eqref{semi}, numerical quadrature is usually needed to compute the exact integral in the practical calculation. 
%To avoid this numerical quadrature error, we slightly modify our numerical schemes. 
%Specifically, the equations \eqref{semi} are replaced by some collocation conditions while \eqref{semi} is substituted by the numerical quadrature formulation. 
%Denote by $g_{\tau,j}$, $1 \leq j \leq n_k$ the collocation points in each $\tau \in T_h $ and $\omega_{\tau,j}$ the corresponding quadrature weights associated with $g_{\tau,j}$ , with $n_k$ the number of collocation points.
%Let
%\begin{equation}
%S_{\tau}=\{g_{\tau,j}: 1\leq j\leq n_k\},\quad S=\{S_{\tau}, \tau \in T_h\}, \quad 	S_{\partial \tau}=S_{\tau} \cap \partial \tau.
%\end{equation}
%Define
%\begin{subequations}
%\begin{align}
%&(f,v)_{\tau}=\int_{\tau} fv \mathrm{d}\tau,\quad (f,v)_{\partial \tau}=\int_{\partial \tau} fv \mathrm{d}\tau, \quad (f,v)=\sum_{\tau \in T_h}	(f,v)_{\tau},\\
%&\langle f,v\rangle _{ \tau}=\sum_{j=1}^{ n_k} fv (g_{\tau,j}),\quad \langle f,v\rangle _{\partial \tau}=\sum_{g_{\tau,j} \in S_{\partial \tau}} fv (g_{\tau,j}), \quad \langle f,v\rangle=\sum_{\tau \in T_h}	\langle f,v\rangle _{ \tau}.
%\end{align}
%\end{subequations}
		
Given any positive function $\varphi(x)$, define
\begin{equation}\label{phi}
a_{\varphi}(\cdot, \cdot )=\sum_{\tau\in T_h} a_{\varphi,\tau}(\cdot, \cdot ),
\end{equation}
where
\begin{equation}\label{bilin}
a_{\varphi,\tau}(u,v)=-( \varphi \nabla u,\nabla v )_{\tau}+( \varphi,\widehat{\partial_n u}v+(u-\{u\})\partial_n v)_{\partial \tau}.
\end{equation}
Then \eqref{semi} can be rewritten as
\begin{subequations}\label{semi1}
\begin{align}
(\partial_t u_h,v )_{\tau}=\chi B a_{\varphi,\tau}(g(u_h),v)-\chi  a_{\varphi,\tau}(c_h,v),\\\label{semi2}
\beta (\partial_t c_h,\theta )_{\tau}=a_{1,\tau}(c_h,\theta)-\alpha(c_h,\theta)_{\tau}+(u_h,\theta)_{\tau}.
\end{align}
\end{subequations}
%In the subsequent part of this paper, our algorithm, and theoretical analysis are perpetually predicated on the  schemes \eqref{semi1} -\eqref{semi2}.
		
\subsection{The fully-discrete numerical scheme}	
     %       Keller-Segel system can be written as a formula using energy function $ E_c[u,c]$ and $	E_e[u,c]$. Gives a time step $h_t>0$, the discrete-time, continuous-space scheme of the Keller-Segel system is designed as follows
%\begin{subequations}\label{energy}
%	\begin{align}
%		\frac{u^{m+1}-u^{m}}{h_t}&=\nabla \cdot (\chi \varphi(u^m)  \nabla \mu_{u}^{m+1}),\\
%		\mu_{u}^{m+1}&=\delta_u E_c(u^{m+1}-c^{m+1})-\delta_u E_e(u^{m}-c^{m}),\\
%		\beta \frac{c^{m+1}-c^{m}}{h_t}&=-\mu_{c}^{m+1},\\
%		\mu_{c}^{m+1}&=\delta_c E_c(u^{m+1}-c^{m+1})-\delta_c E_e(u^{m}-c^{m}),
%	\end{align}
%\end{subequations}
%where $\delta$ denotes the variational derivative.

In this subsection, we propose the first-order implicit schemes for the time discretization.
Given $u_h^{m},c_h^{m}\in V_h^0$, find $u_h^{m+1},c_h^{m+1}\in V_h^0$ such that for all $v,\theta \in  V_h^0$,
\begin{subequations}\label{fully1}	
\begin{align}
(\frac{u_h^{m+1}-u_h^m}{h_t},v)_{\tau}&=\chi B a_{\varphi^m,\tau}(g(u_h^{m+1}),v)-\chi  a_{\varphi^m,\tau}(c_h^{m+1},v),\label{full}\\ 
\beta(\frac{c_h^{m+1}-c_h^m}{h_t},\theta)_{\tau}&=a_{1,\tau} ( c_h^{m+1},\theta) -\alpha( c_h^{m+1},\theta)_{\tau}+(u_h^m,\theta)_{\tau}.\label{full1}
\end{align}
\end{subequations}
Furthermore, we follow the idea of convexity splitting and obtain equivalent
fully-discrete  system,
\begin{subequations}\label{full_energy}
	\begin{align}
		(	\frac{u^{m+1}_h-u^{m}_h}{h_t}, v)_{\tau} & =   \chi a_{\varphi^m,\tau}(\mu_{u}^{m+1},v),\label{eq:1}\\
		\mu_{u}^{m+1}&=\delta_u E_c(u^{m+1}_h,c^{m+1}_h)-\delta_u E_e(u^{m+1}_h,c^{m+1}_h),\label{eq:2}\\
		\beta (\frac{c^{m+1}_h-c^{m}_h}{h_t},\theta)_{\tau}&=-(\mu_{c}^{m+1},\theta ),\label{eq:3}\\
		\mu_{c}^{m+1}&=\delta_c E_c(u^{m+1}_h,c^{m+1}_h)-\delta_c E_e(u^{m+1}_h,c^{m+1}_h)+u^{m+1}_h-u^{m}_h\label{eq:4}.
	\end{align}
\end{subequations}

	\section{Property preservation}\label{sec:property}
    In this section, we study the  properties of the fully discrete numerical scheme, specifically focusing on its mass conservation, positivity preservation, and energy stability. Furthermore, we introduce a novel positivity-preserving limiter to guarantee that the numerical solution maintains its physical positivity throughout the computation.

    The collocation points $g^{\tau}_{i,j}$ are taken as the Gauss-Lobatto points
    in each element $\tau$. i.e.,
    \begin{equation}
    	S_{\tau}=\{ g_{i,j}^{\tau}=(g_i^x,g_j^y): 1\leq i,j \leq k+1 \},
    \end{equation}
    where $g_i^x,g_j^y$
    denote the Gauss-Lobatto points along the $x$ and $y$ directions, respectively.
    %Our subsequent analysis will demonstrate that this special choice of collocation points gives rise to an
   % optimal error estimates for the numerical solution.
    
			\subsection{Mass conservation}
			 By taking $v=1$ in the first equation of \eqref{full}, summing up all elements, using the periodic boundary condition, one can obtain
			\begin{eqnarray*}
				\int_{\Omega} u_h^{m+1}\mathrm{d}\bm{x} =	\int_{\Omega} u_h^{m}\mathrm{d}\bm{x}= \int_{\Omega} u_h^{0}\mathrm{d}\bm{x},
			\end{eqnarray*}
			which means, the fully-discrete numerical scheme \eqref{fully1} is mass conservation.

			\subsection{ Positivity preservation}
		In this subsection, we aim to rigorously demonstrate that the fully discrete numerical scheme preserves positive cell averages. Building on this foundation, we will subsequently develop a positivity-preserving limiter that ensures the numerical solution remains positive throughout the entire domain.  
			
			To facilitate the analysis, we begin with some preliminaries.  The broken space is defined by
	\begin{eqnarray*}
		{\cal H}_h=\{ v\in L^2: v|_{\tau}\in H^1, \forall \tau \in T_h\},
	\end{eqnarray*}
	and 
	\begin{eqnarray*}
		\|v\|_E=\left(\sum_{\tau \in T_h}\int_{\tau} |\nabla v|^2 \mathrm{d}\tau 
        +\sum_{e\in E_h^0}\int_{e} \frac{1}{h_e}[v]^2 \mathrm{d}s\right)^{\frac{1}{2}}.
	\end{eqnarray*}
	
	For the property of the bilinear form $a_{\varphi}(\cdot, \cdot)$ ,  assume that
	\begin{equation*}
		0\leq \varphi_0\leq \varphi\leq \varphi_1,
	\end{equation*}
recalling the defination of $a_{\varphi}$  \eqref{phi}, then	according to \cite{cao2025} and \cite{HLiu2015}, if the parameter of numerical flux $(\beta_0,\beta_1)$ satisfy the stability condition, 
\begin{equation*}
\varphi_0 \beta_0 \geq \varphi_1\Gamma (\beta_1), \quad with \quad
\Gamma(\beta_1)=\sup_{v\in \mathbb{P}_{k-1}(\xi),
\xi\in [-1,1]} 
\frac{2(v(1)-2\beta_1 \partial_\xi v(1))^2}{\int_{-1}^1 v^2(\xi) d\xi},
\end{equation*}
there exists  positive numbers $\gamma_0, \gamma_1$ (dependent on $\varphi_0$ and $\varphi_1$, respectively) such that
	\begin{eqnarray}\label{aphi1}
		a_{\varphi}(v,v)\geq \gamma_0 \|v\|_E^2, \quad 	|a_{\varphi}(u,v)|\leq \gamma_1 \|u\|_E \|v\|_E, \quad \forall v\in V_h.
	\end{eqnarray}
    Define
    	\begin{equation*}
	\tilde{V}_h = \bigl\{ v\in L^2(\Omega) : v|_{\tau}\in\mathbb{P}_k(\tau),\;
	\int_{\Omega} v\,\mathrm{d}\Omega=0 \bigr\}
\end{equation*}
	and let
\begin{equation}
{\mathcal{L}_{\varphi}}:L^2(\Omega)\to\tilde{V}_h,\quad
	a_{\varphi}({\mathcal{L}_{\varphi}} (f),v)=(f,v)\quad \forall v\in\tilde{V}_h.
	\label{a_phi}
\end{equation}
When $\varphi=1$ we write ${\mathcal{L}_{\varphi}}=\mathcal{L}$.

Since $\tilde{V}_h$ has zero mean, Poincaré's inequality and the coercivity of $a_{\varphi}(\cdot,\cdot)$ yield
\begin{equation*}
	\|{\mathcal{L}_{\varphi}} (f)\|_{0}^{2}
	\lesssim |{\mathcal{L}_{\varphi}} (f)|_1^{2}
	\lesssim \frac{\gamma_1}{\gamma_0}\|f\|_0 \|{\mathcal{L}_{\varphi}} (f)\|_0.
\end{equation*}
Applying the standard inverse inequality (constant depends on shape-regularity and $k$) gives
\begin{equation}\label{lphi}
	\|{\mathcal{L}_{\varphi}} (f)\|_{0,\infty}
	\le C h^{-1}\|{\mathcal{L}_{\varphi}} (f)\|_0
	\le \frac{C\gamma_1}{h\gamma_0}\|f\|_0.
\end{equation}
The induced norm is defined by
\begin{equation*}
	\|f\|_{\mathcal{L}_{\varphi}}:=\sqrt{(f,{\mathcal{L}_{\varphi}} (f))}.
\end{equation*}
    Additionally, we have following lemma for subsequent proof (see, e.g. \cite{cao2025}).
			\begin{lemma}
				For function $f,v\in L^2$, let ${\cal{L}}_{\varphi}$ be the operator defined in \eqref{a_phi}. Then we have
				\begin{eqnarray*}
					\frac{1}{2} \frac{\mathrm{d}}{\mathrm{d}s}\|f+sv\|_{{\cal L}_{\varphi}}^{2} 
                    =({\cal{L}}_{\varphi}(f),v)+s(v,{\cal{L}}_{\varphi}(v)).
				\end{eqnarray*}
			\end{lemma}
			
	Furthermore,		for any function $v \in H_h$, we define the cell average of $v$ in each element $\tau$ by
			\begin{eqnarray*}
			\bar	v=\frac{1}{|\tau|}\int_{\tau} v \mathrm{d} \tau.
			\end{eqnarray*}
			
			\begin{theorem}\label{theo3}
			Given $u^m_h, c^m_h \in V_h^0$ with
\begin{equation}
0 < u^m_h(g_{i,j}^\tau) < 1,\quad
c^m_h(g_{i,j}^\tau) > 0,\quad
\|u^m_h\|_{\infty} \le M
\end{equation}
for some $M \ge 0$, all $\tau \in T_h$ and $1 \le i,j \le k+1$,
there exists a unique solution $u_h^{m+1}, c_h^{m+1} \in V_h^0$ to
\eqref{fully1} such that
cell-average positivity: $\bar u_h^{m+1} > 0$,
  and point-wise bounds
        $0 < u_h^{m+1}(g_{i,j}^\tau) < 1$ ,
       $\bar c_h^{m+1}> 0$
        for all nodes $g_{i,j}^\tau$ and all $\tau \in T_h$.

			\end{theorem}
			\begin{proof}
				Denote $a_0=\frac{1}{|\Omega|}\int_{\Omega} u_h^m \mathrm{d}{\bm x}$, $\mu=u-a_0\in \tilde{V}_h$.   The numerical solution of \eqref{fully1} is a minimizer of the discrete energy function
                \begin{equation}
				\begin{aligned}
					J^m (u,c)=&\frac{1}{2h_t}\|u-u^m_h\|_{ L_{\varphi}}^2+\frac{\beta}{2h_t} \|c-c^m_h\|_2^2+B( u,\ln u)+B( 1-u,\ln (1-u))\\
			&	+\frac{1}{2}\|\nabla c\|_2^2+	\frac{1}{2}\alpha\| c\|_2^2+	\frac{1}{2}\|\frac{1}{\gamma} u-\gamma c\|_2^2-
						\frac{1}{\gamma^2}( u,u^m_h) -\gamma^2 ( c,c^m_h) ,
				\end{aligned}
                \end{equation}
				over the admissible set
				\begin{equation}
					A_h:=\{(u,c)\in V_h^0| 0<u(g^{\tau}_{i,j})<1,\quad (u,1)=|\Omega|\,a_0\}.
				\end{equation}

				We can observe that $J(\mu)$  is a strictly convex function over this set.				
               To simplify the subsequent analysis, we reformulate the minimization problem into an equivalent form.
				Consider the function
				\begin{equation}
					\begin{aligned}
						{\cal F}^m(\mu,c)=&	J^m (\mu+a_0,c)\\
						=&\frac{1}{2h_t}\|u-u^m_h\|_{ L_{\varphi}}^2+\frac{\beta}{2h_t} \|c-c^m_h\|_2^2+B(\mu+a_0,\ln (\mu+a_0))\\
                        &+B( 1-\mu-a_0,\ln (1-\mu-a_0))+\frac{1}{2}\|\nabla c\|_2^2+	\frac{1}{2}\alpha\| c\|_2^2\\
						&	+	\frac{1}{2}\|\frac{1}{\gamma} (\mu+a_0)-\gamma c\|_2^2-
						\frac{1}{\gamma^2}( \mu+a_0,u^m_h) -\gamma^2 ( c,c^m_h) ,
					\end{aligned}
				\end{equation}
				defined on the set 
					\begin{equation}
					A_h^*:=\{\mu \in \tilde{V}_h, c\in V_h| 0<(\mu+a_0)(g^{\tau}_{i,j})<1,(\mu,1)=0\}.
				\end{equation}

				Then we will prove that there exists
				a minimizer of ${\cal F}^m(\mu,c)$ over the domain $A_h^*$. That is, for any $\delta > 0$, we consider  the following closed domain
					\begin{eqnarray*}
					A_{h,\delta}^*:=\{\mu \in \tilde{V}_h, c\in V_h^0| \delta\leq (\mu+a_0)(g^{\tau}_{i,j})\leq 1-\delta, (\mu,1)=0\}.
				\end{eqnarray*}
				Since $\tilde{V}_h\times V_h^0$ is finite-dimensional, $A_{h,\delta}^*$ is closed, bounded, convex and hence compact, and ${\cal F}^m$ is strictly convex.  
Therefore a unique minimizer exists in $A_{h,\delta}^*$.  
For sufficiently small $\delta>0$ the interior of $A_{h,\delta}^*$ is non-empty, so the minimizer does not lie on the boundary and satisfies the strict positivity bound.

				We denote by $\mu^*$, $c^*$
				the minimizer of ${\cal F}^m(\mu,c)$ and the minimizer occurs at the boundary of $A_{h,\delta}^*$.
				Assume that $\mu^*(g^{\tau}_{i_0,j_0})+a_0=\delta$. And  Suppose that $\mu^*$
				attains its maximum
				value at the point $g^{\tau}_{i_1,j_1}$. Using the fact $\int_{\Omega} \mu^* \mathrm{d}\bm{x}=0$  yields that $a_0\leq \mu^*(g^{\tau}_{i_1,j_1})+a_0\leq 1-\delta$.
				
			Then we consider the directional derivative, for any $\phi \in V_h^0$,
				\begin{equation}
                \begin{aligned}
			&	\frac{\mathrm{d}}{\mathrm{d}s} {\cal F}^m(\mu^*+s\phi,c^*)|_{s=0}\\
            =&\frac{1}{h_t}({\cal L}_{\varphi}(\mu^*+a_0-u^m_h),\phi)
					+B( \ln (\mu^*+a_0)-\ln (1-(\mu^*+a_0)) ,\phi) \\
			& +	( \frac{1}{\gamma^2} \mu^*+a_0- c^*,\phi )-
					\frac{1}{\gamma^2}( u^m_h,\phi).
                    \end{aligned}
				\end{equation}
	
    Let $p^{\tau}_{k,l}\in V_h^0$ be the Lagrange basis function satisfying
$p^{\tau}_{k,l}(g^{\tau}_{i,j})=\delta_{ik}\delta_{jl}$.
 Now we choose $\phi=p^{\tau}_{k,l}$ in the last equation to derive
                \begin{equation*}
				\begin{aligned}
				\frac{\mathrm{d}}{\mathrm{d}s} {\cal F}^m(\mu^*+s\phi,c^*)|_{s=0}=&\frac{1}{h_t} {\cal L}_{\varphi}(\mu^*+a_0-u^m_h)(g^{\tau}_{i,j})
					\\&+B ( \ln (\mu^*+a_0)-\ln (1-(\mu^*+a_0)) )(g^{\tau}_{i,j})
					\\&+	 \frac{1}{\gamma^2} (\mu^*+a_0- c^*)(g^{\tau}_{i,j}) -
					\frac{1}{\gamma^2} u^m_h (g^{\tau}_{i,j}).
				\end{aligned}
                \end{equation*}
                Here we use nodal quadrature so that the inner product reduces to the nodal value up to a common multiplicative factor $|\tau|$, which is canceled by dividing both sides of the equation. 
			
                Taking $\phi=p^{\tau}_{i_0,j_0}$ and $\phi=p^{\tau}_{i_1,j_1}$ respectively, then
                subtracting the two equations yields the desired relation
\begin{equation}
\begin{aligned}
0={}&\frac{1}{h_t}\Bigl[{\cal L}_{\varphi}(\mu^*+a_0-u_h^m)(g^{\tau}_{i_0,j_0})
                                 -{\cal L}_{\varphi}(\mu^*+a_0-u_h^m)(g^{\tau}_{i_1,j_1})\Bigr]\\
&+B\Bigl[\ln(\mu^*+a_0)-\ln\bigl(1-(\mu^*+a_0)\bigr)\Bigr](g^{\tau}_{i_0,j_0})\\
&-B\Bigl[\ln(\mu^*+a_0)-\ln\bigl(1-(\mu^*+a_0)\bigr)\Bigr](g^{\tau}_{i_1,j_1})\\
&+\frac{1}{\gamma^2}\Bigl[(\mu^*+a_0-c^*)(g^{\tau}_{i_0,j_0})
                                 -(\mu^*+a_0-c^*)(g^{\tau}_{i_1,j_1})\Bigr]\\
&-\frac{1}{\gamma^2}\Bigl[u_h^m(g^{\tau}_{i_0,j_0})-u_h^m(g^{\tau}_{i_1,j_1})\Bigr].
\end{aligned}
\end{equation}
			
				Since $\mu^*(g^{\tau}_{i_0,j_0})+a_0=\delta$, and $\mu^*(g^{\tau}_{i_0,j_0})+a_0\geq a_0$, we have
				\begin{eqnarray*}
					( \ln (\mu^*+a_0)-\ln (1-(\mu^*+a_0)) )|^{g^{\tau}_{i_0,j_0}}_{g^{\tau}_{i_1,j_1}}\leq \ln \frac{\delta}{1-\delta}-\ln \frac{a_0}{1-a_0},
				\end{eqnarray*}
				and 
				\begin{eqnarray*}
					\mu^*(g^{\tau}_{i_0,j_0})-\mu^*(g^{\tau}_{i_1,j_1})<0.
				\end{eqnarray*}
				For the numerical solution $u^m_h$ at the previous time step, the a priori assumption $\|u^m_h\|\leq M$
				indicates that
				\begin{eqnarray*}
					-2M \leq u^m_h(g^{\tau}_{i_0,j_0})-u^m_h(g^{\tau}_{i_1,j_1})\leq 2M.
				\end{eqnarray*}
				Using \eqref{lphi}, one can obtain 
				\begin{eqnarray*}
					({\cal L}_{\varphi}(\mu^*+a_0-u^m)|_{g^{\tau}_{i_1,j_1}}^{g^{\tau}_{i_0,j_0}}\leq 2 C_0  h^{-1},
				\end{eqnarray*}
				where Here $C_0 =C\frac{\gamma_1}{\gamma_0}$
				with $C$ the same as that in \eqref{lphi}.
				Denote $C_1=\max(|c^*(g_{i_0,j_0}^{\tau})|,|c^*(g_{i_1,j_1}^{\tau})|)$. Because of the choice of the admissible set $A_{h,\delta}^*$, we have
				 \begin{eqnarray*}
				 	-2C_1<- c^*|^{g^{\tau}_{i_0,j_0}}_{g^{\tau}_{i_1,j_1}}\leq 2 C_1.
				 \end{eqnarray*}
				Define $D_0=h_t^{-1}2 C_0  h^{-1}+2C_1+\frac{2M}{\gamma^2} $.
				Note that $D_0$ is a constant when $h$ and $h_t$ is fixed. For any fixed $h$ and $h_t$ , we may choose a special
				$\delta  > 0$ small enough so that
				\begin{eqnarray*}
					B \ln \frac{\delta}{1-\delta}-B\ln \frac{a_0}{1-a_0}+D_0<0,
				\end{eqnarray*}
				and thus 
				\begin{eqnarray*}
						\frac{\mathrm{d}}{\mathrm{d}s} {\cal F}^m (\mu^*+s\phi,c^*)|_{s=0}<0,\quad \forall \phi\in V_h^0.
				\end{eqnarray*}
				This contradicts the assumption that ${\cal F}^m(\mu,c)$ has a minimum at $(\mu^*,c^*)$
				. Therefore, the global minimum
				of ${\cal F}^m(\mu,c)$ over $A_{h,\delta}^*$ could only possibly occur at an interior point, for $\delta > 0$ sufficiently small, which is equivalent to the numerical solution of \eqref{fully1}.
			 The existence of the numerical solution is established. Furthermore, since $J^m$
			  is strictly convex function over $A_h$, the uniqueness analysis for this
			 numerical solution is straightforward.
				
				Since
				\begin{eqnarray*}
					\bar u^{m+1}_h=\frac{1}{\tau}\sum_{i,j=1}^{k+1} \omega^{\tau}_{i,j} u_h^{m+1} (g^{\tau}_{i,j})>0,
				\end{eqnarray*}
				then the positivity preserving of the numerical solution at Gauss-Lobatto points and for the cell average is established.
				
				As for function $c$,  we proceed by contradiction. 
              Assume that the cell-average of $ c_h^{m+1} $ on some element $ \tau_0 $ is negative:
\begin{equation*}
\bar{c}_{\tau_0}^{m+1} = \frac{1}{|\tau_0|} \int_{\tau_0} c_h^{m+1} < 0.
\end{equation*}
Take the test function $ \theta \equiv 1 $ on $ \tau_0 $ and $ \theta \equiv 0 $ elsewhere in the fully-implicit  scheme \eqref{full1} for $ c $. This yields the exact cell-average equation
\begin{equation}\label{c_avg}
\frac{\beta}{h_t} \left( \bar{c}_{\tau_0}^{m+1} - \bar{c}_{\tau_0}^m \right) + \alpha \bar{c}_{\tau_0}^{m+1} = \bar{u}_{\tau_0}^m + \frac{1}{|\tau_0|} \sum_{e \subset \partial \tau_0} \int_e \widehat{\partial_n c_h^{m+1}} \mathrm{d}s,
\end{equation}
 Summing over all edges of $ \tau_0 $ gives
\begin{equation*}
 \frac{1}{|\tau_0|} \sum_{e \subset \partial \tau_0} \int_e \widehat{\partial_n c_h^{m+1}} \mathrm{d}s \leq  0.
\end{equation*}

Inserting this estimate into equation \eqref{c_avg} we obtain
\begin{equation*}
\frac{\beta}{h_ t} \left( \bar{c}_{\tau_0}^{m+1} - \bar{c}_{\tau_0}^m \right) + \alpha \bar{c}_{\tau_0}^{m+1} \leq \bar{u}_{\tau_0}^m.
\end{equation*}

Under the induction hypothesis $ \bar{c}_{\tau_0}^m \geq 0 $  and the assumed $ \bar{c}_{\tau_0}^{m+1} < 0 $, the left-hand side is strictly negative, while the right-hand side is non-negative. This yields the contradiction.
Therefore $ \bar{c}_{\tau_0}^{m+1} \geq 0 $ on every element. Since the limiter introduced in subsection 4.3 subsequently squeezes the polynomial back into $ [0, \infty) $ without changing the cell-average, all Gauss-Lobatto nodal values remain non-negative, and the positivity of $ c_h^{m+1} $ is established.

				The proof is completed.
			\end{proof}
		
			\subsection{Limiter}
By Theorem~\ref{theo3} the cell averages are positive, but the high-order polynomial $u_h|_\tau\in\mathbb{P}^k(x) \times \mathbb{P}^k(y)$ may still take negative or 
 greater than $1$ values.  
We therefore apply the 'Zhang-Shu positivity-preserving limiter' \cite{ZS2010} element-wise:

\begin{equation*}
\tilde u_h(\bm x)\big|_\tau
= \bar u_h + \theta\,(u_h(\bm x)-\bar u_h),
\quad
\theta=\min\!\Bigl\{1,\;
\frac{\bar u_h}{\bar u_h-\min_{\bm x} u_h(\bm x)},\;
\frac{1-\bar u_h}{\max_{\bm x} u_h(\bm x)-\bar u_h}
\Bigr\},
\end{equation*}
where the min/max are taken over the Gauss–Lobatto nodes of $\tau$.  
This guarantees

\begin{equation*}
0\le \tilde u_h(\bm x)\le 1, \quad \forall\bm x\in\tau,
\qquad
\frac{1}{|\tau|}\int_\tau \tilde u_h(\bm x)\mathrm{d}\bm x=\bar u_h,
\end{equation*}

and 
\begin{equation*}
|\tilde{u}_h(\bm x)-u_h(\bm x)|\leq C_k \|u_h-u\|_{\infty}. \end{equation*}
The above estimate demonstrates that the constructed solution $\tilde{u}_h$ maintains the high-order accuracy of the original polynomial, conserves the cell averages exactly, and ensures positivity at all points within each element. 

Similarly,  we apply the following limiter,
\begin{equation*}
\tilde{c}_h(\bm x)=\bar c_h+\theta(c_h(\bm x)-\bar c_h),\quad
\theta=\min\{ 1,\frac{\bar c_h}{\bar c_h-\min_{\bm x} c_h(\bm x)}\}
\end{equation*}
to ensure that $c_h$ are non-negative.
       \subsection{Energy stability}

         In this section,  we  show that fully-discrete numerical scheme is energy stable.
           \begin{theorem}
           	For $m\geq 1$, the numerical scheme \eqref{full_energy} is  energy stable providing $h_t$ is sufficiently  small, i.e.
           	\begin{equation*}
           			E[u^{m+1}_h,c^{m+1}_h]\leq 	E[u^{m}_h,c^{m}_h],
           	\end{equation*}
            with the following definition of discrete energy 
            \begin{equation*}
                E[u^{m}_{h},c^{m}_{h}]
                =\int_\Omega \{B[u^{m}_{h}\ln u^{m}_{h}+(1-u^{m}_{h})\ln(1-u^{m}_{h})]-u^{m}_{h}c^{m}_{h}+\frac{1}{2}(|\nabla c^{m}_{h}|^2+\alpha (c^{m}_{h})^2) \}\mathrm{d}\bm{x}.
            \end{equation*}
           \end{theorem}
           
          \begin{proof}
          	We consider fully discrete scheme of energy \eqref{full_energy}.
          	First, by choosing $v=\mu_u^{m+1}$ in \eqref{eq:1} to obtain
          	\begin{equation}\label{energy1}
          		\begin{aligned}
          			(u^{m+1}_h-u^{m}_h,\mu_u^{m+1})=h_t\chi a_{\varphi^m,\tau}(\mu_u^{m+1},\mu_u^{m+1}).
          		\end{aligned}
          	\end{equation}
          	
          		Second, by choosing $\theta=\mu_c^{m+1}$ in \eqref{eq:2} to obtain
          	\begin{equation}\label{energy2}
          		\begin{aligned}
          			(c^{m+1}_h-c^{m}_h,\mu_c^{m+1})=-\frac{h_t}{\beta} (\mu_c^{m+1},\mu_c^{m+1}).
          		\end{aligned}
          	\end{equation}
          	Noticing that
 %         	\begin{equation}
%          		\begin{aligned}
  %        			E(u^{m+1}_h,c^{m+1}_h)-	E(u^{m}_h,c^{m}_h)
 %         			=(E_c(u^{m+1}_h,c^{m+1}_h)-	E_c(u^{m}_h,c^{m}_h))-(E_e(u^{m+1}_h,c^{m+1}_h)-	E_e(u^{m}_h,c^{m}_h)),
    %      		\end{aligned}
    %      	\end{equation}
         	 $E_c[u, c]$ and $E_e[u, c]$ are both
          	convex with respect to $u$ and $c$, then use \eqref{energy1} and \eqref{energy2} to obtain
       
            \begin{equation}\label{energy3}
\begin{aligned}
&E[u_h^{m+1},c_h^{m+1}]-E[u_h^m,c_h^m]\\
\le &\bigl(\delta_u E_c[u_h^{m+1},c_h^{m+1}]-\delta_u E_e[u_h^{m+1},c_h^{m+1}],u_h^{m+1}-u_h^m\bigr)\\
&\quad +\bigl(\delta_c E_c[u_h^{m+1},c_h^{m+1}]-\delta_c E_e[u_h^{m+1},c_h^{m+1}],c_h^{m+1}-c_h^m\bigr)\\
= &\bigl(\mu_u^{m+1},u_h^{m+1}-u_h^m\bigr)
   +\bigl(\mu_c^{m+1}-(u_h^{m+1}-u_h^m),c_h^{m+1}-c_h^m\bigr)\\
= &h_t\chi\,a_{\varphi,\tau}(\mu_u^{m+1},\mu_u^{m+1})
   -\frac{h_t}{\beta}\|\mu_c^{m+1}\|^2_0
   -\bigl(u_h^{m+1}-u_h^m,c_h^{m+1}-c_h^m\bigr),
\end{aligned}
\end{equation}

    For the cross term $(u^{m+1}_h-u^{m}_h,c^{m+1}_h-c^m_h)$,       
          	 by choosing $v=u^{m+1}_h-u^m_h$ in \eqref{eq:1}   to obtain
             \begin{equation*}
                 \|(u_h^{m+1}-u_h^m)\|_0\lesssim h_t \gamma_1 \chi h^{-1} \|\mu_u^{m+1}\|_E,
             \end{equation*}
             then, choosing  $\theta=u^{m+1}_h-u^m_h$ in \eqref{eq:3}, using the Cauchy-Schwarz inequality and the above bound
             \begin{equation}\label{stale2}
             \begin{aligned}
                (u^{m+1}_h-u^{m}_h,c^{m+1}_h-c^m_h)&=-h_t \beta^{-1} (\mu_c^{m+1}, u^{m+1}_h-u^{m}_h)\\
              &  \lesssim 
                 h_t \beta^{-1} \|\mu_c^{m+1}\|_0  \|u^{m+1}_h-u^{m}_h\|_0\\
                & \leq \tilde{C}  h_t^2 \beta^{-1} \gamma_1 \chi h^{-1}\|\mu_c^{m+1}\|_0  
                 \|\mu_u^{m+1}\|_E,
             \end{aligned}
             \end{equation}
        where  $\tilde{C}$ is a constant independent of $h$ and $h_t$,  depending only on the mesh regularity and $\varphi$.

     Substituting \eqref{stale2} into \eqref{energy3} and using Young's inequality, we obtain
\begin{equation*}
\begin{aligned}
E[u^{m+1}_h,c^{m+1}_h]-E[u^{m}_h,c^{m}_h]
&\le -h_t \varphi_0 \chi\|\mu_u^{m+1}\|_E^2
-\frac{h_t}{\beta}\|\mu_c^{m+1}\|_E^2
+\frac{h_t^2 \chi \tilde{C}}{\beta h}\|\mu_u^{m+1}\|_E^2,
\end{aligned}
\end{equation*}
where we use $\|\mu_c^{m+1}\|_0 \lesssim \|\mu_c^{m+1}\|_E$.
Under the CFL condition
\begin{equation*}
h_t\le \frac{\varphi_0 \beta h}{\tilde{C}},
\end{equation*}
the last term is absorbed by the first two, yielding
\begin{equation*}
E(u^{m+1}_h,c^{m+1}_h)-E(u^{m}_h,c^{m}_h)\le 0.
\end{equation*}
          \end{proof} 

            \section{Error estimate}\label{sec:error}
           
           \subsection{ Preliminaries}

           For error estimate, we prepare some theoretical analysis.
           Define $I_h : H_h \rightarrow V_h^0$ the Gauss-Lobatto interpolation satisfying
           \begin{equation*}
           	I_h v(g^{\tau}_{i,j})=v(g^{\tau}_{i,j}), \quad \forall v\in H_h, \quad 1\leq i, j\leq k+1,\quad \tau \in T_h.
           \end{equation*}
           According to \cite{cao2025}, the  approximation properties are hold:
          \begin{equation*}
          	\|v-I_h v\|_0\lesssim h^{k+1} \|v\|_{k+1}, \quad |\nabla(v-I_h v), \nabla w|\lesssim h^{k+1} \|v\|_{k+2} \|\nabla w\|_0.
          \end{equation*}
          If $\beta_1=\frac{1}{2k(k+1)}$, then (see, e.g., \cite{cao2017})
          \begin{equation*}
          	|{a}(u-I_h u,v)|\lesssim h^{k+1} |u|_{k+2} \|v\|_E,
          \end{equation*}
         thus, noticing that $\varphi$ is bounded \eqref{aphi1}, then we have,
         \begin{equation}\label{eq1}
         	|{a}_{\varphi}(u-I_h u,v)|\lesssim h^{k+1} |u|_{k+2} \|v\|_E.
         \end{equation}

To establish error estimates at time $t^{m+1}$, we assuming that the exact solution $u,c$ are smooth enough,   and depending on the physical properties of $u,c$.   Assume the exact solution satisfies
\begin{eqnarray}\label{smooth_u}
    \delta_0\le u\le 1-\delta_0,\quad
\|u\|_{r,\infty}\le M_0,\quad
\|c\|_{r,\infty}\le M_0,
\end{eqnarray}
where $r\ge k+1$ and $\delta_0,M_0>0$ are constants.
For sufficiently small $h$, the interpolation error estimate
\begin{equation*}
\|I_h u - u\|_{0,\infty}\le C h^{k+1}\|u\|_{k+1,\infty}\le \frac{\delta_0}{2}
\end{equation*}
yields
\begin{equation}
\|I_h u\|_{0,\infty}\ge \frac{\delta_0}{2},\quad
\|I_h u\|_{1,\infty}\le  M_0.
\end{equation}
The same bounds hold for $I_h c$.
We assume the a-priori estimate (to be proved later)
\begin{eqnarray}\label{priori}
    \|u_h^m - I_h u^m\|_E\lesssim h_t + h^{k+1},\quad
\|c_h^m - I_h c^m\|_E\lesssim h_t + h^{k+1}.
\end{eqnarray}
Under the inverse inequalities 
\begin{equation*}
\|v_h\|_{1,\infty}\lesssim h^{-1}\|v_h\|_E,\quad
\|v_h\|_{0,\infty}\lesssim  |\ln h|^{1/2}\|v_h\|_E,
\end{equation*}
and assuming $h_t=O(h)$, we obtain
\begin{equation*}
\|u_h^m\|_{1,\infty}\lesssim 1,\quad
\|u_h^m - I_h u^m\|_{0,\infty}\lesssim |\ln h|^{1/2}(h_t + h^{k+1}).
\end{equation*}
For sufficiently small $h$ and $h_t=O(h)$, we have
\begin{equation}\label{est_umh}
    \|u_h^m\|_{0,\infty}\ge \frac{\delta_0}{4},\quad
\|c_h^m\|_{0,\infty}\ge \frac{\delta_0}{4}.
\end{equation}

 \subsection{Error estimates}
	In the rest of this paper, we will use the following notations
	\begin{eqnarray*}
		e_v=v-v_h,\quad \eta_v=v-I_h v,\quad \xi_v=I_h v-v_h, \quad v=u,c.
	\end{eqnarray*}
	For all $v,\theta \in V_h^0$, noticing that the exact solutions satisfy
\begin{subequations}
    \begin{align}
        (\frac{u^{m+1}-u^m}{h_t},v )= & \chi B {a}_{\varphi^{m+1}}(g(u^{m+1}),v)-\chi {a}_{\varphi^{m+1}}(c^{m+1},v) \nonumber \\ 
            +&(\frac{u^{m+1}-u^m}{h_t}-\partial_t u^{m+1},v ),\\
			\beta (\frac{c^{m+1}-c^m}{h_t},\theta )= & {a}_{1}(c^{m+1},\theta)-\alpha( c^{m+1},\theta)+( u^{m},\theta) \nonumber \\ 
            +&\beta(\frac{c^{m+1}-c^m}{h_t}-\partial_t c^{m+1},\theta ).
    \end{align}
\end{subequations}

	Then we have the following error equations,
	\begin{subequations}
    \begin{align}
	 (\frac{\xi_u^{m+1}-\xi_u^m}{h_t},v)=&\chi B {a}_{\varphi^m}(g(u^{m+1}),v)-\chi B {a}_{\varphi^{m}_h}(g(u^{m+1}_h),v)\nonumber\\
	& -\chi  {a}_{\varphi^m}(c^{m+1},v) +\chi {a}_{\varphi^m_h}(c^{m+1}_h,v)+I_1(v),
    \label{err1}\\
	 \beta (\frac{\xi_c^{m+1}-\xi_c^m}{h_t},\theta )=&{a}_{1}(e_c^{m+1},\theta)-\alpha( \xi_c^{m+1},\theta)+(\xi_u^{m},\theta)+I_2(\theta),
    \label{err2}
    \end{align}
	\end{subequations}
	where
	\begin{subequations}\label{I1-I2}
	    \begin{align}
	        I_1(v)=&
		\chi B {a}_{\varphi^{m+1}-\varphi^{m}}(g(u^{m+1}),v)-\chi B {a}_{\varphi^{m+1}-\varphi^{m}}(c^{m+1},v)
		\nonumber\\
		+&(\frac{u^{m+1}-u^m}{h_t}-\partial_t u^{m+1},v )- (\frac{\eta_u^{m+1}-\eta_u^m}{h_t},v),\\
		I_2(\theta)=&\beta 	(\frac{c^{m+1}-c^m}{h_t}-\partial_t c^{m+1},\theta )\nonumber\\
		-&\beta (\frac{\eta_c^{m+1}-\eta_c^m}{h_t},\theta)
		- \alpha( \eta_c^{m+1},\theta)+(\eta_u^{m},\theta).   
	    \end{align}
	\end{subequations}

	Using the Cauchy-Schwarz inequality, we have for $w = u$ or $c$,
	\begin{equation}
    \begin{aligned}
		&|(\frac{w^{m+1}-w^m}{h_t}-\partial_t w^{m+1},v )-(\frac{\eta_w^{m+1}-\eta_w^m}{h_t},v)| \lesssim (h_t+ h^{k+1})\|v\|_0,\\
		&| {a}_{\varphi^{m+1}-\varphi^{m}}(w^{m+1},v)| \lesssim h_t \|v\|_E,\\
    \end{aligned}
	\end{equation}
	which yields
	\begin{eqnarray}\label{est_I}
		|I_1(v)|\lesssim (h_t+ h^{k+1})\|v\|_E, \quad |I_2(\theta)|\lesssim (h_t+ h^{k+1})\|\theta\|_E,\quad \forall v, \theta \in  V_h^0.
	\end{eqnarray}

\begin{lemma}\label{um1cm1}
		Assume the smoothness estimate \eqref{smooth_u} and the a-priori bound \eqref{priori} hold.
Let $\epsilon>0$ be a constant independent of $h$ and $h_t$ and suppose
\begin{equation}\label{lemma_res}
\frac{h_t}{h}\bigl(h_t+h^{k+1}+\|\xi_u^m\|_0+\|\xi_c^m\|_0\bigr)\le \epsilon,
\end{equation}
with $\epsilon\le \epsilon_0$ sufficiently small so that
\begin{equation*}
M_2\ge \frac{\beta}{2},\quad
M_3\ge \frac{h_t\chi\gamma_1}{2},
\end{equation*}
where $M_2,M_3$ are the coefficients defined in \eqref{coeff_M}.
Then the numerical solutions satisfy the uniform separation
\begin{gather}\label{uhm1}
\frac{3\delta_0}{4}\le \|u_h^{m+1}\|_{0,\infty}\le 1-\frac{3\delta_0}{4},\\
\frac{3\delta_0}{4}\le \|c_h^{m+1}\|_{0,\infty}\le M_0-\frac{3\delta_0}{4}.
\end{gather}
\end{lemma}
		
	\begin{proof}
	    Choosing $\theta=\xi_c^{m+1}$ in \eqref{err2} yields
        \begin{equation*}
            \begin{aligned}
              & \beta(\xi_c^{m+1},\xi_c^{m+1})-h_t {a}_{1}(e_c^{m+1},\xi_c^{m+1})\\
               =&
               \beta(\xi_c^{m},\xi_c^{m+1})-h_t\alpha( \xi_c^{m+1},\xi_c^{m+1})+h_t(\xi_u^{m},\xi_c^{m+1})+h_tI_2(\xi_c^{m+1}),
            \end{aligned}
        \end{equation*}
        where $I_2$ defined in \eqref{I1-I2}.
        Then using Young's inequality, we have
        \begin{equation}\label{lem_c}
            \begin{aligned}
             & \beta \|\xi_c^{m+1}\|_0^2+
              h_t {\bf C}\|\xi_c^{m+1}\|_E^2\\
              \leq & \beta \frac{2}{\epsilon} \|\xi_c^{m+1}\|_0^2+ 
              \beta \frac{\epsilon}{8} \|\xi_c^{m}\|_0^2
             - h_t \alpha \|\xi_c^{m+1}\|_0^2
             +h_t\frac{2}{\epsilon} \|\xi_u^{m}\|_0^2+ 
             h_t \frac{\epsilon}{8} \|\xi_c^{m+1}\|_0^2+h_t I_2(\xi_c^{m+1}),
            \end{aligned}
        \end{equation}
   where $\bf C$ is a constant independent with $h$ and $h_t$.    
        On the other hand,  taking $v=\xi_u^{m+1}$ in \eqref{err1} to obtain
        \begin{align}\label{lem_u}
        &	\|\xi_u^{m+1}\|_0^2 
        	-h_t\chi B {a}_{\varphi^m}(g(u^{m+1}),\xi_u^{m+1})+h_t\chi B {a}_{\varphi^{m}_h}(g(u^{m+1}_h),\xi_u^{m+1})\nonumber\\
        	& +h_t\chi  {a}_{\varphi^m}(c^{m+1},v) -h_t\chi
        	  {a}_{\varphi^m_h}(c^{m+1}_h,\xi_u^{m+1}) \nonumber \\
        	  =& 
        	  \|\xi_u^{m+1}\|_0^2 
        	  -h_t \chi B {a}_{\varphi^m_h}(g(e_u^{m+1}),\xi_u^{m+1})-h_t \chi B {a}_{e_\varphi^m}(g(u^{m+1}),\xi_u^{m+1})\nonumber\\
        	&  + h_t \chi  {a}_{\varphi^m_h}(e_c^{m+1},\xi_u^{m+1})+h_t \chi  {a}_{e_\varphi^m}(c^{m+1},\xi_u^{m+1})\nonumber\\
        	=&(\xi_u^{m},\xi_u^{m+1})+h_tI_1(\xi_u^{m+1})
        \end{align}

        Using \eqref{aphi1}, Young's inequality,  and the bounds established in Lemma  \ref{um1cm1}, we obtain 
        \begin{equation}\label{lem1}
        	 {a}_{\varphi^m_h}(g(e_u^{m+1}),\xi_u^{m+1})\leq \gamma_1 \|\xi_u^{m+1}\|_E^2,
        \end{equation}
       \begin{equation}\label{lem2}
       	{a}_{e_\varphi^m}(g(u^{m+1}),\xi_u^{m+1}) \leq \gamma_1 C_0 \|\xi_u^{m+1}\|_E^2,
       \end{equation}
        \begin{equation}\label{lem3}
       	{a}_{\varphi^m_h}(e_c^{m+1},\xi_u^{m+1})\leq  \gamma_1  \frac{2}{\epsilon}\|\xi_c^{m+1}\|_E^2  +\gamma_1  \frac{\epsilon}{8}\|\xi_u^{m+1}\|_E^2,
       \end{equation}
        \begin{equation}\label{lem4}
       	{a}_{e_\varphi^m}(c^{m+1},\xi_u^{m+1}) \leq   M_0 \gamma_1  \|\xi_u^{m+1}\|_E^2.
       \end{equation}
       \begin{equation}\label{lem5}
       	(\xi_u^{m},\xi_u^{m+1})\leq \frac{2}{\epsilon}\|\xi_u^{m}\|_0^2  +  \frac{\epsilon}{8}\|\xi_u^{m+1}\|_0^2.
       \end{equation}
      Then a substitution of \eqref{lem1}-\eqref{lem5} into \eqref{lem_u} and combing with \eqref{lem_c} reveals that
        \begin{equation}\label{coeff_M}
            \begin{aligned}
                 M_1\|\xi_u^{m+1}\|_0^2+ M_2 \|\xi_c^{m+1}\|_0^2 +M_3 \|\xi_u^{m+1}\|_E^2
                 +M_4 \|\xi_c^{m+1}\|_E^2
                 \leq  C_1\|\xi_u^{m}\|_0^2+ C_2 \|\xi_c^{m}\|_0^2,
            \end{aligned}
        \end{equation}
        where 
        \begin{equation}
        \begin{aligned}
        & M_1=1-\frac{\epsilon}{8}, \\
        & M_2=\beta-\beta \frac{2}{\epsilon}+h_t \alpha-\beta\frac{\epsilon}{8}, \\
        & M_3=-h_t \chi B\gamma_1-h_t \chi B\gamma_1 C_0+h_t \chi \gamma_1 \frac{\epsilon}{8}+h_t \chi M_0\gamma_1, \\
        & M_4=h_t {\bf C}+h_t \chi B \gamma_1\frac{2}{\epsilon}, \\
        & C_1=h_t \frac{2}{\epsilon}+\frac{2}{\epsilon}, \\ 
        & C_2=\beta \frac{\epsilon}{8}. 
        \end{aligned}
        \end{equation}
Provided $\epsilon\le \epsilon_0$ with $\epsilon_0$ small enough so that
\begin{equation*}
M_3\ge \frac{h_t\chi\gamma_1}{2},\quad
M_2\ge \frac{\beta}{2},
\end{equation*}
 with \eqref{est_I}, we obtain
\begin{equation*}
\|\xi_u^{m+1}\|_0+\|\xi_c^{m+1}\|_0
\le  {\mathcal{C}}h_t(h_t+h^{k+1}+\|\xi_u^m\|_0+\|\xi_c^m\|_0).
\end{equation*}

We choose $h_t=O(h)$ and $h\le h_0(\delta_0,\mathcal C)$, so that
\begin{equation*}
\mathcal C\frac{h_t}{h}(h_t+h^{k+1})\le \frac{\delta_0}{4}.
\end{equation*}
Using the inverse inequality  and the priori bound  \eqref{priori}  yields
\begin{equation*}
\|\xi_u^{m+1}\|_{0,\infty}+\|\xi_c^{m+1}\|_{0,\infty}
\le \mathcal{C}\frac{h_t}{h}(h_t+h^{k+1}+\|\xi_u^m\|_0+\|\xi_c^m\|_0)
\le \frac{\delta_0}{4}.
\end{equation*}
Consequently,
\begin{equation*}
\frac{3\delta_0}{4}\le \|u_h^{m+1}\|_{0,\infty}\le 1-\frac{3\delta_0}{4},\quad
\frac{3\delta_0}{4}\le \|c_h^{m+1}\|_{0,\infty}\le M_0-\frac{3\delta_0}{4}.
\end{equation*}
This completes the proof.
\end{proof}

        \begin{lemma}
    Under the assumptions \eqref{smooth_u} and \eqref{priori}, there exist positive constants $c_0,c_1$ independent of $h$ and $h_t$ such that
\begin{gather}\label{lemma2}
c_0\|\xi_u^{m+1}\|_E^2
\le a_{\varphi^m}(g(e_u^{m+1}),\xi_u^{m+1})
\le c_1 \|\xi_u^{m+1}\|_E^2. 
\end{gather}
\end{lemma}
            \begin{proof}
                Using the property of $a_{\varphi}(\cdot, \cdot)$ in \eqref{aphi1} ,  we have
                \begin{equation*}
                    \begin{aligned}
                          a_{\varphi^m}(g(e_u^{m+1}),\xi_u^{m+1})\lesssim \|g(e_u^{m+1})\|_E \|\xi_u^{m+1}\|_E.
                    \end{aligned}
                \end{equation*}
                Noticing that $g(u)=\ln u-\ln(1-u)$,  by using the Taylor expansion , there exists a $u_{\xi_1}$ between $u^{m+1}$ and $u^{m+1}_h$ such that
                \begin{eqnarray*}
                    \ln u^{m+1}-\ln u^{m+1}_h= \frac{1}{u_{\xi_1}} e_u^{m+1} =\frac{1}{u_{\xi_1}} (\xi_u^{m+1}+\eta_u^{m+1}).
                \end{eqnarray*}
                For function $\ln(1-u)$ , there exists a $u_{\xi_2}$ to ensure the above equality holds. 
                Then we have
                \begin{align*}
                      a_{\varphi^m}(g(e_u^{m+1}),\xi_u^{m+1})&= a_{\varphi^m}(\frac{\xi_u^{m+1}+\eta_{u}^{m+1}}{u_{\xi_1}},\xi_u^{m+1})-a_{\varphi^m}(\frac{\xi_u^{m+1}+\eta_{u}^{m+1}}{u_{\xi_2}},\xi_u^{m+1}) \nonumber \\
                      &= a_{\varphi^m}(\frac{\xi_u^{m+1}}{u_{\xi_1}},\xi_u^{m+1})-a_{\varphi^m}(\frac{\xi_u^{m+1}}{u_{\xi_2}},\xi_u^{m+1}). 
                \end{align*}
                In light of  \eqref{uhm1} and \eqref{smooth_u}, we have
                \begin{eqnarray*}
                    \frac{3\delta_0}{4}\leq u_{\xi_i}\leq 1-\frac {3\delta_0}{4}.
                \end{eqnarray*}
                Together with \eqref{est_umh} yields
                \begin{eqnarray*}
                    a_{\varphi^m}(\frac{\xi_u^{m+1}}{u_{\xi_i}},\xi_u^{m+1})\geq \frac{3\delta_0 \gamma_0}{4(1-\frac{3\delta_0}{4})}\|\xi_u^{m+1}\|^2_E.
                \end{eqnarray*}
                Consequently,

                \begin{eqnarray*}
                    a_{\varphi^m}(g(e_u^{m+1}),\xi_u^{m+1}))\geq \frac{3\delta_0 \gamma_0}{4(1-\frac{3\delta_0}{4})}\|\xi_u^{m+1}\|^2_E. 
                \end{eqnarray*}

                On the other hand, by the continuity of $a_{\varphi}(\cdot,\cdot)$ and the separation bound, we have
\begin{equation*}
	a_{\varphi^m}(g(e_u^{m+1}),\xi_u^{m+1})
	\le \gamma_1\|g(e_u^{m+1})\|_E\|\xi_u^{m+1}\|_E
	\le \frac{\gamma_1}{\delta_0}\|\xi_u^{m+1}\|_E^2,
\end{equation*}
which completes our proof. 
\end{proof}

	Now we are ready to present the error estimates for the fully discretization numerical solution.
	\begin{theorem}
		Let that  $u^{m+1}_h, c^{m+1}_h$  be the solution of \eqref{full}-\eqref{full1} with the coefficient $\beta_1=\frac{1}{2k(k+1)}$
		for the DDG discretization. Then 
		\begin{eqnarray}
	\label{result1}	&	\|u^{m+1}-u^{m+1}_h\|_0+\|I_h u^m-u^{m}_h\|_E\lesssim h_t+h^{k+1},\\
	\label{result2}	&\|c^{m+1}-c^{m+1}_h\|_0+\|I_h c^m-c^{m}_h\|_E \lesssim h_t+h^{k+1}.
		\end{eqnarray}
	\end{theorem}
	
	\begin{proof}
		First, by choosing $\theta=\xi_c^{m+1}$ in \eqref{err2} yields,
		\begin{equation*}
			\frac{\beta}{2 h_t}(\|\xi_c^{m+1}\|_0^2-\|\xi_c^{m}\|_0^2+\|\xi_c^{m+1}-\xi_c^{m+1}\|_0^2)={a}_{1}(e_c^{m+1},\xi_c^{m+1})-\alpha( \xi_c^{m+1},\xi_c^{m+1})+(\xi_u^{m},\xi_c^{m+1})+I_2(\xi_c^{m+1}).
		\end{equation*}
Then using \eqref{est_I}, the Cauchy-Schwarz inequality and the a-priori assumption \eqref{priori}, we have 
		\begin{equation*}
		\begin{aligned}
			&	\frac{\beta}{ h_t}(\|\xi_c^{m+1}\|_0^2-\|\xi_c^{m}\|_0^2+\|\xi_c^{m+1}-\xi_c^{m}\|_0^2)\\
			\lesssim& \|\xi_c^{m+1}\|_E^2+ \alpha \|\xi_c^{m+1}\|_0^2+\|\xi_u^{m}\|_0 \|\xi_c^{m+1}\|_0+(h_t+ h^{k+1})\|\xi_c^{m+1}\|_E\\
			\lesssim &\|\xi_c^{m+1}\|_E^2+ \alpha \|\xi_c^{m+1}\|_0^2+\|\xi_u^{m}\|_0^2 +\|\xi_c^{m+1}\|_0^2+(h_t+ h^{k+1})^2,
		\end{aligned}
	\end{equation*}
	and thus,
		\begin{equation}\label{est}
			\|\xi_c^{m+1}\|_0^2-\|\xi_c^{m}\|_0^2+\|\xi_c^{m+1}\|_E \lesssim h_t\|\xi_u^{m}\|_0^2+h_t(h_t+h^{k+1})^2.
		\end{equation}
		
			By choosing $v=\xi_u^{m+1}$ in \eqref{err1},	using the Cauchy-Schwarz inequality and the a-priori assumption \eqref{priori}, we  can bound the term ${a}_{e_\varphi^m}(g(u^{m+1}),\xi_u^{m+1})$ as follows:
		\begin{equation*}
			{a}_{e_\varphi^m}(g(u^{m+1}),\xi_u^{m+1}) \lesssim \|u^{m+1}\|_{1,\infty} \|e_\varphi^m\|_0 \|\xi_u^{m+1}\|_E\lesssim (h_t+h^{k+1})\|\xi_u^{m+1}\|_E,
		\end{equation*}
		similarly,
		\begin{equation*}
			{a}_{e_\varphi^m}(c^{m+1},\xi_u^{m+1}) \lesssim \|c^{m+1}\|_{1,\infty} \|e_\varphi^m\|_0 \| \xi_u^{m+1}\|_E\lesssim (h_t+h^{k+1})\|\xi_u^{m+1}\|_E.
		\end{equation*}
		
	 Then using \eqref{lemma2} and \eqref{est_I} yields
		\begin{equation*}
			\begin{aligned}
			&	\frac{1}{2 h_t}(\|\xi_u^{m+1}\|_0^2-\|\xi_u^{m}\|_0^2+\|\xi_u^{m+1}-\xi_u^{m}\|_0^2)\\
				\lesssim & \|\xi_u^{m+1}\|_E^2+\|\xi_u^{m+1}\|_E \|\xi_c^{m+1}\|_E+(h_t+ h^{k+1})\|\xi_u^{m+1}\|_E\\
				\lesssim& \|\xi_u^{m+1}\|_E^2+\|\xi_c^{m+1}\|_E^2+ \|\xi_u^{m+1}\|_E^2+(h_t+ h^{k+1})\|\xi_u^{m+1}\|_E.
			\end{aligned}
		\end{equation*}
		 Substituting \eqref{est} into above inequality, we have
		 \begin{equation*}
		 	\|\xi_u^{m+1}\|_0^2+\|\xi_u^{m+1}\|_E ^2         +\|\xi_c^{m+1}\|_0^2+\|\xi_c^{m+1}\|_E^2 \lesssim h_t(\|\xi_u^{m}\|_0^2+\|\xi_c^{m}\|_0^2)+h_t(h_t+h^{k+1})^2.
		 \end{equation*}
	Under the CFL condition $C h_t \le 1/2$, discrete Gronwall's inequality yields
\begin{equation*}
\|\xi_u^{m+1}\|_0 + \|\xi_u^{m+1}\|_E \lesssim h_t + h^{k+1},
\end{equation*}
which together with the approximation properties $\|u^{m+1} - I_h u^{m+1}\|_0 \lesssim h^{k+1}$ and $\|I_h u^{m+1} - u_h^{m+1}\|_E \lesssim h^{k+1}$ gives the desired estimate \eqref{result1}.

Similarly, for \eqref{result2}, we apply elliptic-projection error $\|I_h c^{m+1} - c^{m+1}\|_E \lesssim h^{k+1}$ to obtain
\begin{equation*}
\|\xi_c^{m+1}\|_0 + \|I_h c^{m+1} - c_h^{m+1}\|_E \lesssim h_t + h^{k+1}.
\end{equation*}
The proof is complete.
\end{proof}

    \begin{remark}
        Our scheme is conditionally energy stable. 
        In fact, it is easy to modify it to the unconditional energy-stable scheme 
        by applying the convex splitting method
        only if we replace $u^{m}$ with $u^{m+1}$ in \eqref{full1}, 
        which will result in a coupled discrete scheme. 
        Or we can also add stabilized factors $\tau_{u}(u^{m+1}-u^{m})$ and $\tau_{c}(c^{m+1}-c^{m})$ into the discrete chemical potential $\mu_{u}^{m+1}$ and $\mu_{c}^{m+1}$, respectively, to maintain the unconditional energy stability, where $\tau_{u}$ and $\tau_{c}$ are two constants and independent of time and space steps. 
        However, it is necessary to rewrite the original energy to a modified version with the stabilization included. 
        Both these two schemes are feasible with no extra difficulties in the proof of positivity and error estimate. 
        We leave the details to the interested readers. 
    \end{remark}

\section{Numerical results}\label{sec:results}

In this section, we present a few examples to demonstrate the accuracy  of DDG methods. Furthermore, positivity-preserving  and mass conservation of  the proposed schemes are numerically validated. In the tests, uniform meshes are used, which are constructed by equally dividing  the computational domain into $N$ cells ($N\times N$  rectangles for 2D problem).  To diminish the time discretization error, 
		  step size is $dt=0.01 h_{min}^2$. 
\begin{example}\label{eg1}
Convergence rate test for one-dimensional problems.

We consider a modified KS system obtained by adding suitable source terms so that an exact solution is available.
On the one-dimensional domain  $\Omega=[0,2\pi]$ with periodic boundary conditions, the modified equations read 
		\begin{equation*}
		\begin{aligned}
			u_t&=\nabla(\chi \varphi(u) \nabla (Bg(u)-c))+f_u,\\
			\beta c_t&=\Delta c-\alpha c+u+f_c,\\
		\end{aligned}
	\end{equation*}
with periodic boundary condition for $u$ and $c$,	where $f_u$ and $f_c$ are source functions.
	The initial solutions are
	\begin{equation*}
		u(x,0)=0.3\sin(x)+0.5,\quad c(x,0)=\sin(x)+2.
	\end{equation*}
	This system admits the exact solution
	\begin{equation*}
	u(x,t)=\exp(-t)(0.3\sin(x)+0.5),\quad c(x,t)=\exp(-t)(\sin(x)+2).
    \end{equation*}
    \end{example}
    Parameters are selected as $\chi=0.1$, $\alpha=0.2$, $B=0.2$, $\beta=0.01$. 
Table \ref{table1} and Table \ref{table2} presents the errors and convergence orders at $ t = 0.01 $ for the numerical flux parameters $(\beta_0, \beta_1) = \left(\frac{7}{6}, 0\right)$ using $ \mathbb{P}_1 $ and $\mathbb{P}_2$ polynomials, respectively. We observe $(k+1)$-order accuracy for both $ u_h^m $ and $ c_h^m $.
\begin{table}[ht]
\begin{center}
\begin{tabular}{c|cccc}
\hline
$N$ & $\mathrm{err}_u$ & $\mathrm{rate}_u$ & $\mathrm{err}_c$ & $\mathrm{rate}_c$ \\
\hline
8&	1.97E-02&	-- &	4.47E-02&	-- \\
		16&	4.91E-03&	2.00& 	1.06E-02&	2.08 \\
		32&	1.18E-03&	2.05 &	2.61E-03&	2.02 \\
		64&	2.59E-04&	2.19 &	6.51E-04&	2.01 \\
		128	&6.15E-05&	2.07 &	1.62E-04&	2.00 \\
\hline
\end{tabular}
\end{center}
\caption{Accuracy test with $\mathbb{P}_1(x)$ quadratic polynomials for Example \ref{eg1}.}
\label{table1}
\end{table}
\begin{table}[ht]
\begin{center}
\begin{tabular}{c|cccc}
\hline
$N$ & $\mathrm{err}_u$ & $\mathrm{rate}_u$ & $\mathrm{err}_c$ & $\mathrm{rate}_c$ \\
\hline
4 & 9.72\text{E-}03 & -- & 3.16\text{E-}02 &--\\
8 & 1.25\text{E-}03 & 2.96 & 3.91\text{E-}03 & 3.01 \\
16 & 1.57\text{E-}04 & 2.99 & 4.59\text{E-}04 & 3.09 \\
32 & 1.96\text{E-}05 & 3.00 & 5.52\text{E-}05 & 3.05 \\
64 & 2.48\text{E-}06 & 2.98 & 6.88\text{E-}06 & 3.01\\
\hline
\end{tabular}
\end{center}
\caption{Accuracy test with $\mathbb{P}_2(x)$ quadratic polynomials for Example \ref{eg1}.}
\label{table2}
\end{table}

	Next, we investigate the effectiveness of numerical methods in preserving physical properties.
    We compute the numerical solution by the proposed DDG method on a fine mesh with $N_x=64$ and using the second-order ($k = 2$) piecewise polynomials.  The calculation time $t = 0.5$.
     The left of    Figure \ref{fig1d} vividly illustrates the temporal evolution of discrete energy, clearly confirming the energy dissipation property of the system. 
The right one displays the rough estimates of mass conservation throughout the simulation. The fact that the mass remains nearly constant over time numerically verifies the property of mass conservation.
Furthermore, Figure \ref{fig1d1} offers a detailed analysis of the maximum and minimum values of $ u $ and the minimum value of $ c $ throughout the simulation. This information underscores the robustness of the proposed method, which effectively maintains essential physical constraints, such as non-negativity and boundedness, across the entire computation process.

	\begin{figure}[!ht]
		\centering
		\subfigure{
			\begin{minipage}{0.48\textwidth}
				\includegraphics[width=1.0\textwidth]{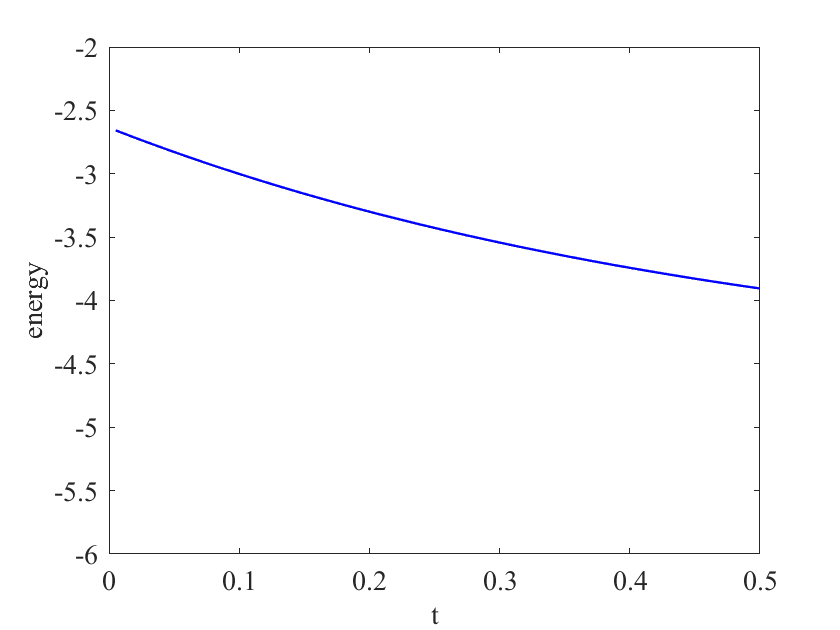}
			\end{minipage}
		}	
		\subfigure{
			\begin{minipage}{0.48\textwidth}
				\includegraphics[width=1.0\textwidth]{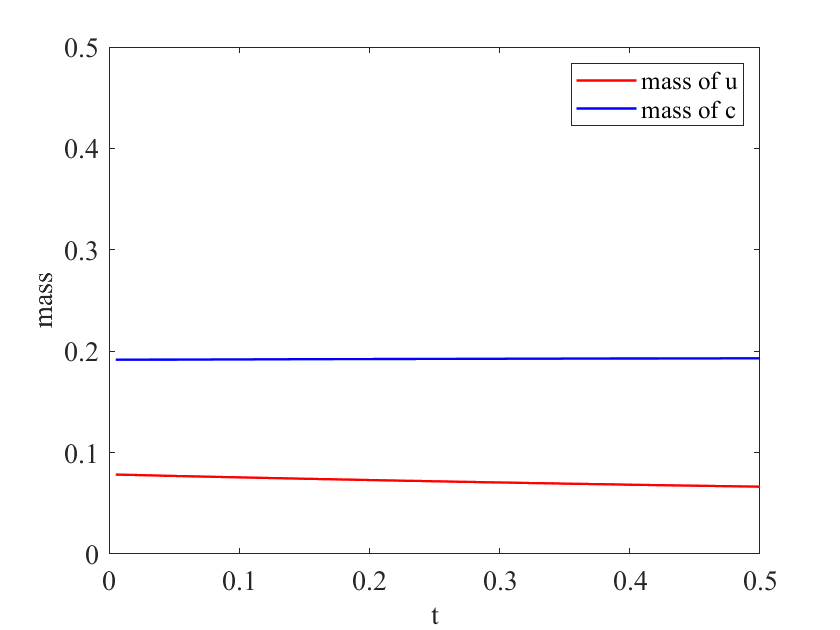}
			\end{minipage}
		}

		\caption{ Time evolution of free energy (left) and mass of $u$ and $c$ (right) for Example  \ref{eg1} .}
		\label{fig1d}
	\end{figure}

		\begin{figure}[!ht]
		\centering
		
		\subfigure{
			\begin{minipage}{0.48\textwidth}
				\includegraphics[width=1.0\textwidth]{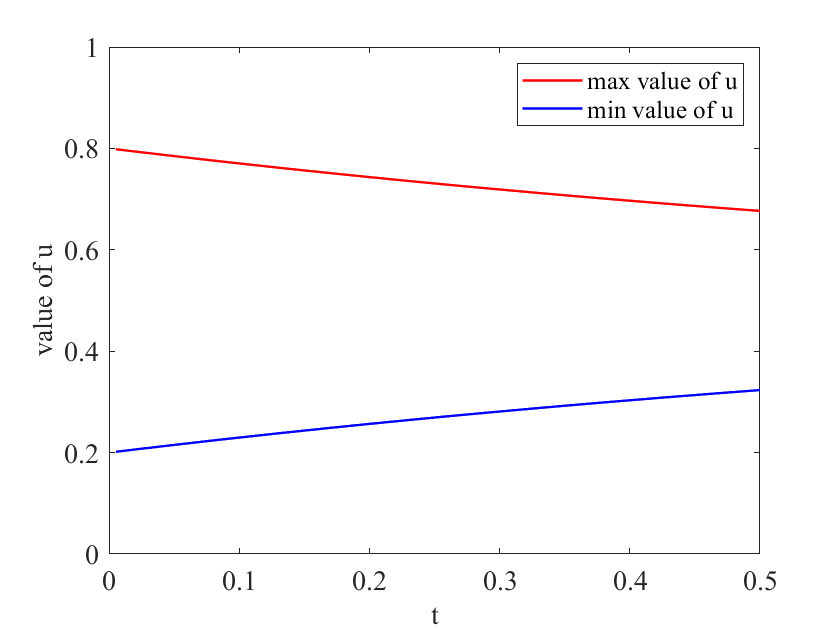}
			\end{minipage}
		}	
		\subfigure{
			\begin{minipage}{0.48\textwidth}
				\includegraphics[width=1.0\textwidth]{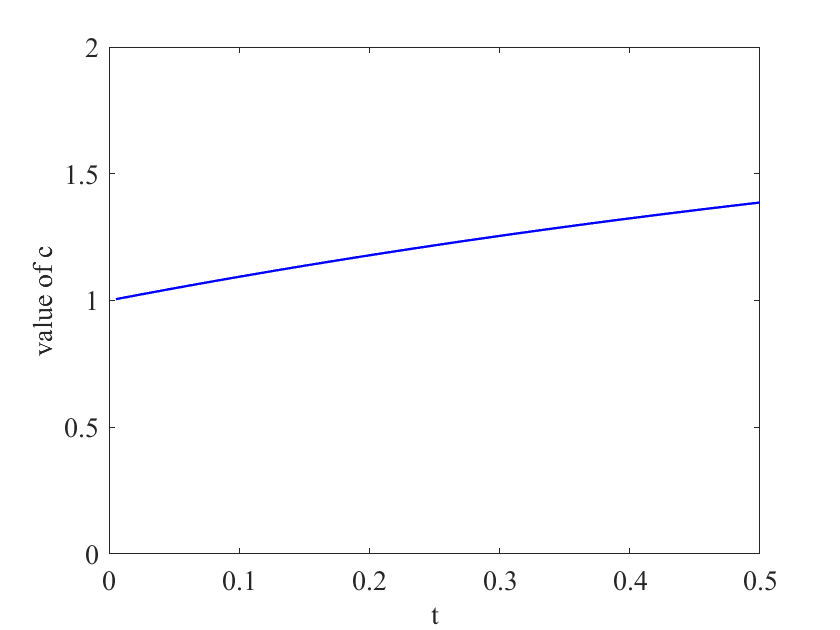}
			\end{minipage}
		}

		\caption{ Time evolution of maximum and minimum values of $u$ (left) and minimum value of $c$  (right) Example \ref{eg1}.}
		\label{fig1d1}
	\end{figure}

\begin{example}\label{eg2}
Convergence rate test for two-dimensional problems.       

In this example, we consider the KS system in a rectangle domain $\Omega=[0, 2\pi]^ 2$ with source terms $f_u(x,y,t)$ and $f_c(x,y,t)$ added.
The initial solutions are chosen as
\begin{equation*}
u(x,y,0)=0.3\sin(x)\cos(y)+0.5,\qquad
c(x,y,0)=\sin(x)\cos(y)+2.
\end{equation*}
The corresponding exact solutions are
\begin{equation*}
u(x,y,t)=\exp(-t)(0.3\sin(x)\cos(y)+0.5),\qquad
c(x,y,t)=\exp(-t)(\sin(x)\cos(y)+2),
\end{equation*}
with periodic boundary conditions enhanced. 
\end{example}

\begin{table}[!ht]
\begin{center}
\begin{tabular}{c|cccc}
\hline
$N\times N$ & $\mathrm{err}_u$ & $\mathrm{rate}_u$ & $\mathrm{err}_c$ & $\mathrm{rate}_c$ \\
\hline
$10\times 10$	&1.95E-02&	-- &	6.50E-02&	--\\
$20\times 20$	&4.90E-03&	1.99 &	1.63E-02&	1.99 \\
$30\times30$	&2.20E-03&	1.99& 	7.30E-03	&2.00 \\
$40\times40$&	1.20E-03	&1.96 	&4.10E-03	&2.00 \\
$50\times50$	&7.80E-04&	1.93 &	2.60E-03&	2.00\\ 
\hline
\end{tabular}
\end{center}
\caption{Accuracy test with $\mathbb{P}_1(x)\times \mathbb{P}_1(y) $ quadratic polynomials for Example \ref{eg2}.}
\label{err_2d}
\end{table}

\begin{table}[!ht]
\begin{center}
\begin{tabular}{c|cccc}
\hline
$N$ & $\mathrm{err}_u$ & $\mathrm{rate}_u$ & $\mathrm{err}_c$ & $\mathrm{rate}_c$ \\
\hline
 $4\times 4$& 1.58E-02 & -- & 5.27E-02 & -- \\
$8 \times8$& 2.02E-03 & 2.97 & 6.73E-03 & 2.97 \\
$16\times 16$& 2.54E-04 & 2.99 & 8.46E-04 & 2.99 \\
$32\times 32$& 3.22E-05 & 2.98 & 1.06E-04 & 3.00 \\
\hline
\end{tabular}
\end{center}
\caption{Accuracy test with $\mathbb{P}_2(x)\times \mathbb{P}_2(y) $ quadratic polynomials for Example \ref{eg2}.}
\label{err_2d_p2}
\end{table}

The discrete errors $L^2$ norm and their respective orders of $u$ and $c$ at $t = 0.01$, are reported in Table \ref{err_2d} and \ref{err_2d_p2}. 
From these tables, one can see that the solution numerically converges to the exact solution with the (optimal) second order which confirms the theoretical results predicted by our convergence analysis. 
The temporal evolution of discrete energy, mass and maximum and minimum value of $u$ and $c$ are presented in Figure \ref{fig2d} and \ref{fig1d2}. 
The figures numerically confirm that the proposed numerical method can preserve the physical properties, including energy decay, mass conservation, and the boundness of $u$ and $c$. 

\begin{figure}[!ht]
\centering
\subfigure{
\begin{minipage}{0.48\textwidth}
\includegraphics[width=1.0\textwidth]{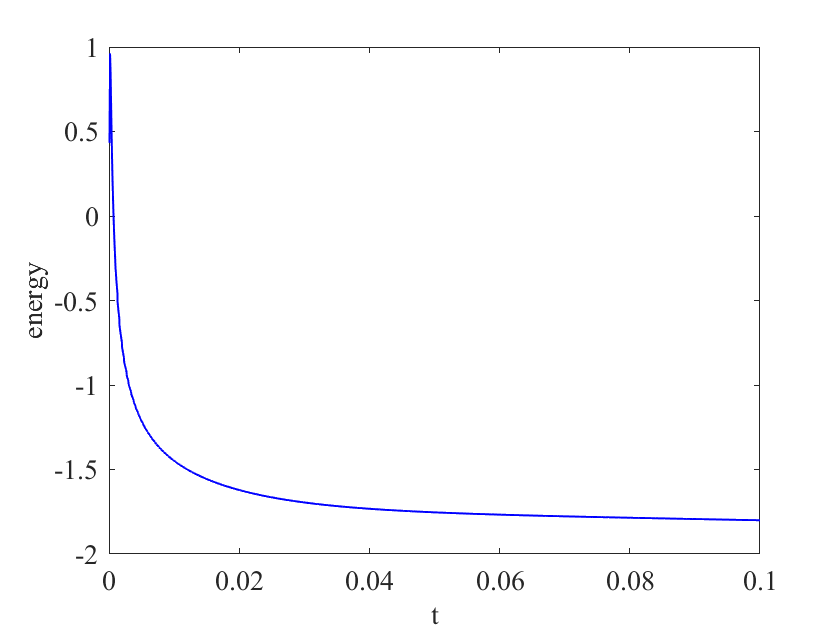}
\end{minipage}
}	
\subfigure{
\begin{minipage}{0.48\textwidth}
\includegraphics[width=1.0\textwidth]{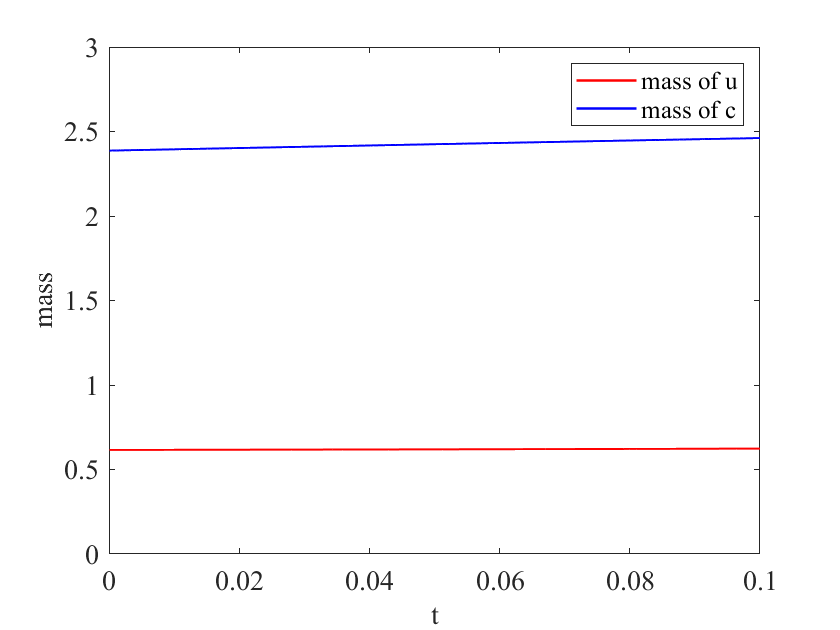}
\end{minipage}
}
\caption{Time evolution of free energy (left) and mass of $u$ and $c$ (right) for Example \ref{eg2}.}
\label{fig2d}
\end{figure}

\begin{figure}[!ht]
\centering
\subfigure{
\begin{minipage}{0.48\textwidth}
\includegraphics[width=1.0\textwidth]{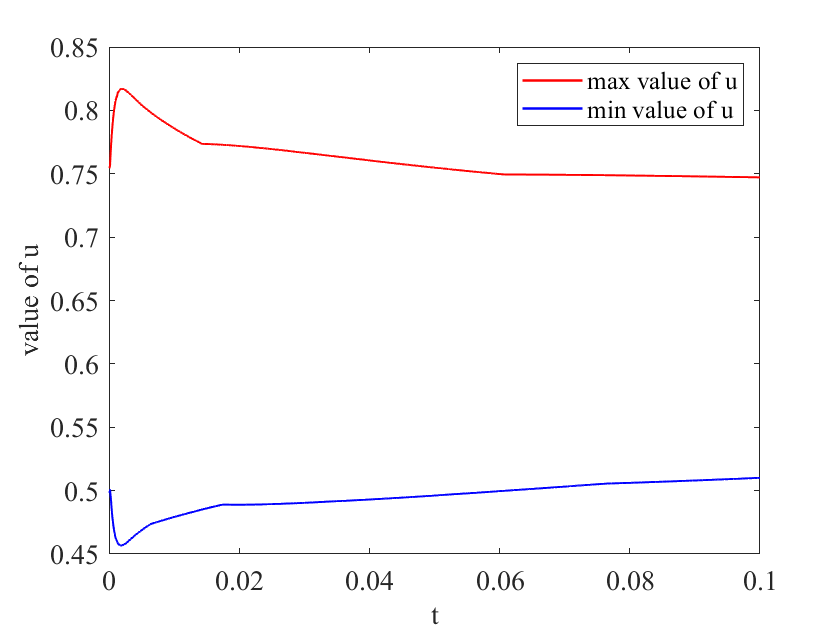}
\end{minipage}
}
\subfigure{
\begin{minipage}{0.48\textwidth}
\includegraphics[width=1.0\textwidth]{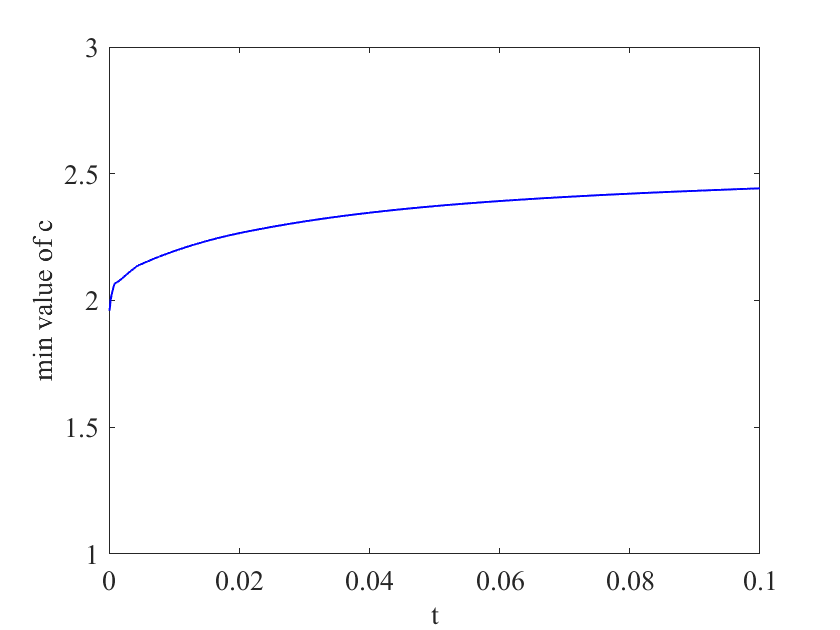}
\end{minipage}
}
\caption{Time evolution of maximum and minimum values of $u$ (left) and minimum value of $c$  (right) Example \ref{eg2}.}
\label{fig1d2}
\end{figure}

\begin{example}\label{eg3}The equilibrium state of two-dimensional problems. 

We test our scheme for the 2D KS system with parameters $\chi=0.1,\alpha=0.02, B=0.5,\beta=1$ on the domain $[0,1]^2$.
The initial solution are
\begin{equation*}
u(x,y,0)=0.8\exp(-\sqrt{(x-0.5)^2+(y-0.5)^2}/0.05),\end{equation*}
\begin{equation*}
c(x,y,0)=0.1(1+0.1\sin(2\pi x)\sin(2\pi y)).
\end{equation*}
\end{example}
\begin{figure}[!ht]
	\centering
	\subfigure[$u_h(x,y,0)$]{
		\begin{minipage}{0.33\textwidth}
			\includegraphics[width=1.0\textwidth]{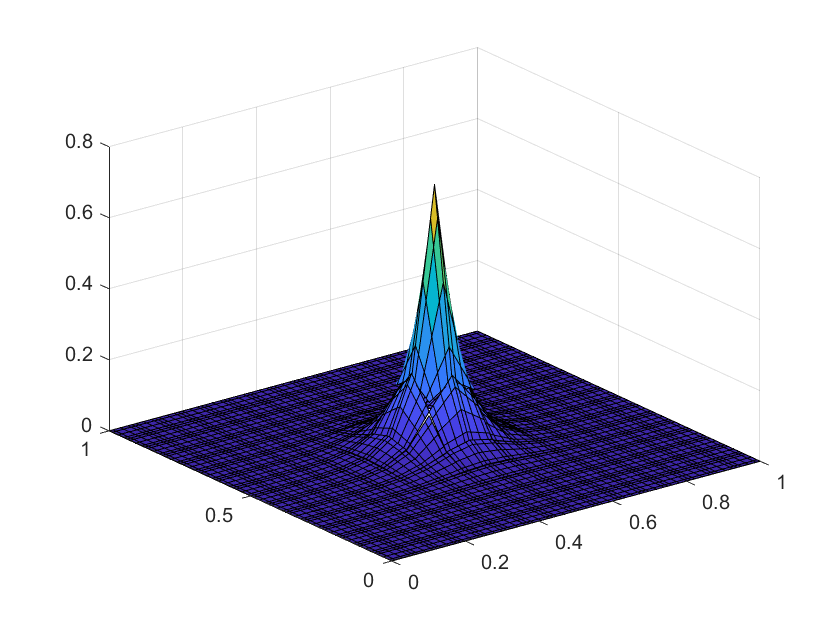}
		\end{minipage}
	}
    \hskip -0.3cm
		\subfigure[$u_h(x,y,0.01)$]{
		\begin{minipage}{0.33\textwidth}
			\includegraphics[width=1.0\textwidth]{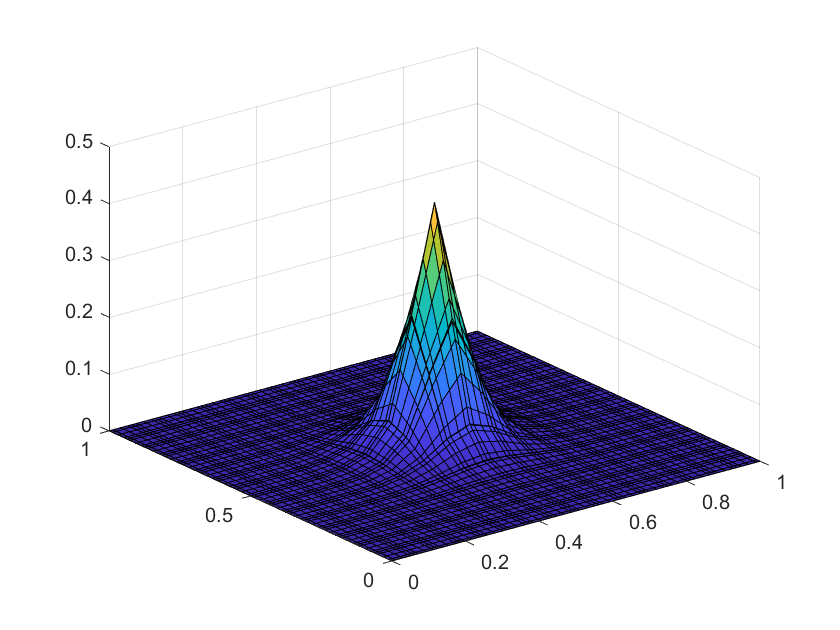}
		\end{minipage}
	}
    \hskip -0.3cm
	\subfigure[$u_h(x,y,0.03)$]{
		\begin{minipage}{0.33\textwidth}
			\includegraphics[width=1.0\textwidth]{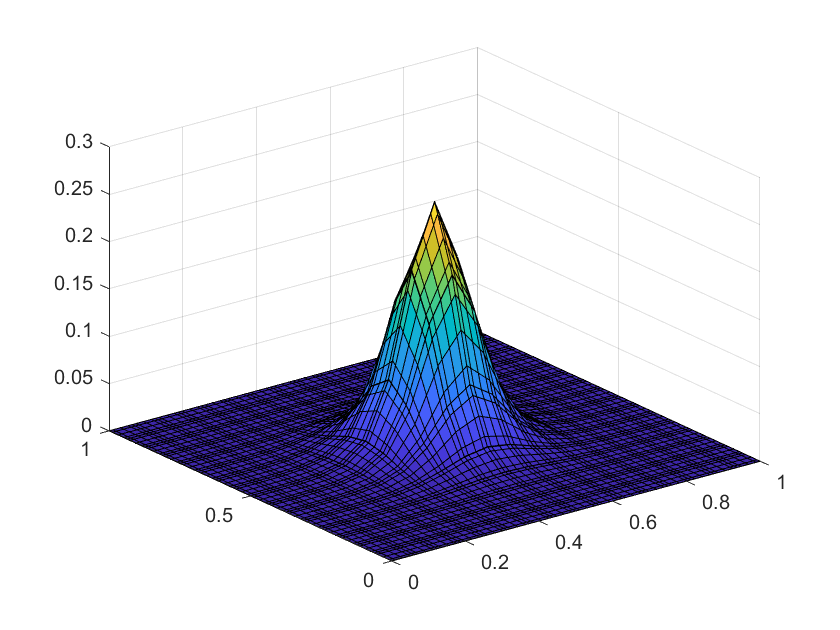}
		\end{minipage}
	}
    \hskip -0.3cm
	\subfigure[$u_h(x,y,0.05)$]{
		\begin{minipage}{0.33\textwidth}
			\includegraphics[width=1.0\textwidth]{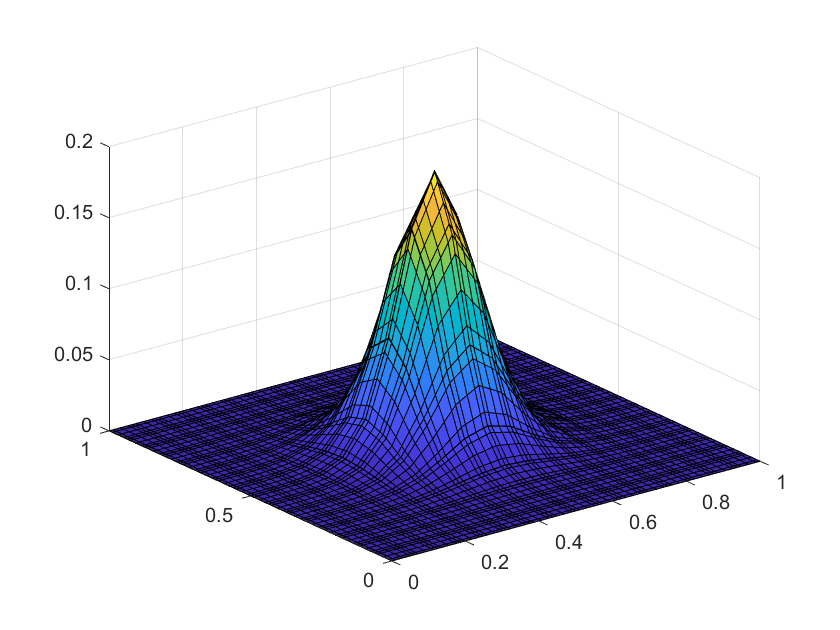}
		\end{minipage}
	}
    \hskip -0.3cm
	\subfigure[$u_h(x,y,0.1)$]{
		\begin{minipage}{0.33\textwidth}
			\includegraphics[width=1.0\textwidth]{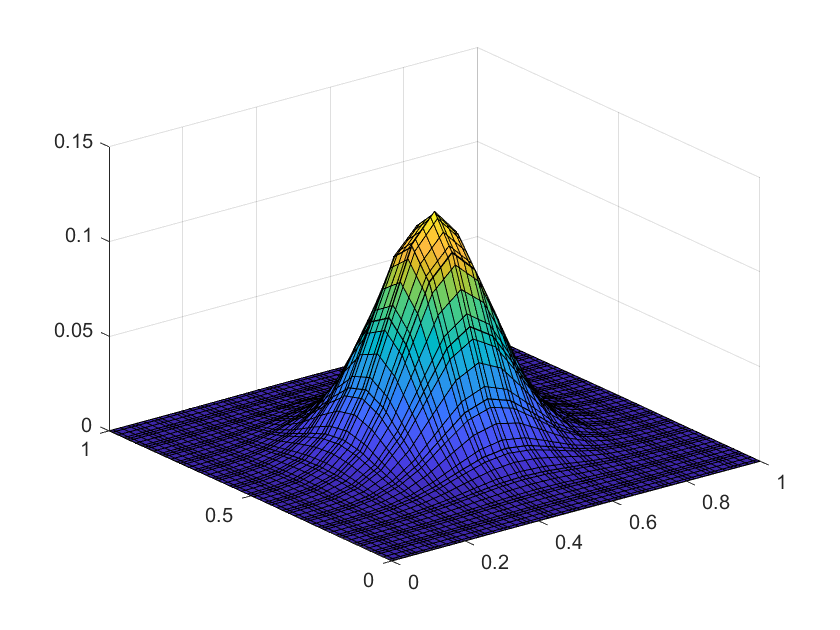}
		\end{minipage}
	}
    \hskip -0.3cm
	\subfigure[$u_h(x,y,0.5)$]{
		\begin{minipage}{0.33\textwidth}
			\includegraphics[width=1.0\textwidth]{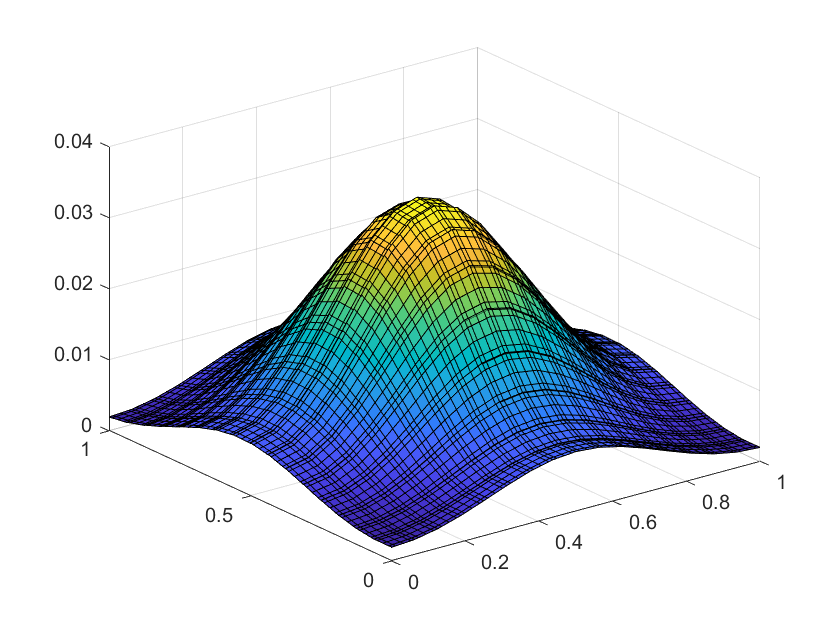}
		\end{minipage}
	}
    \hskip -0.3cm
	\subfigure[$c_h(x,y,0)$]{
		\begin{minipage}{0.33\textwidth}
			\includegraphics[width=1.0\textwidth]{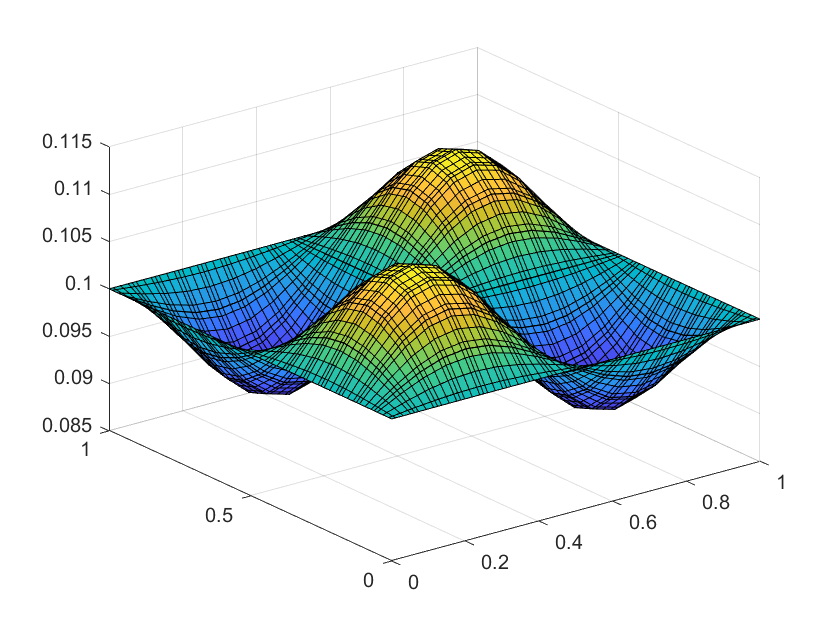}
		\end{minipage}
	}
    \hskip -0.3cm
		\subfigure[$c_h(x,y,0.01)$]{
			\begin{minipage}{0.33\textwidth}
				\includegraphics[width=1.0\textwidth]{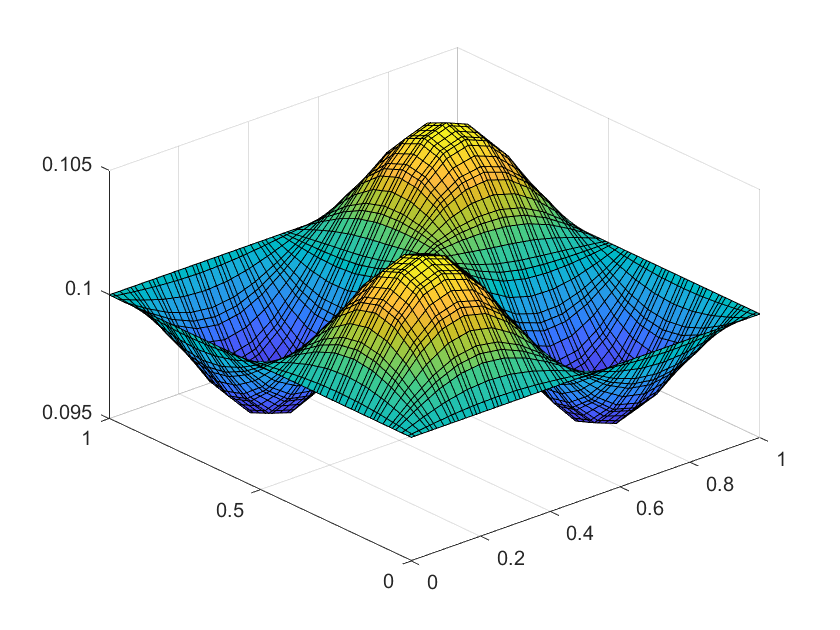}
			\end{minipage}
	}
    \hskip -0.3cm
	\subfigure[$c_h(x,y,0.03)$]{
		\begin{minipage}{0.33\textwidth}
			\includegraphics[width=1.0\textwidth]{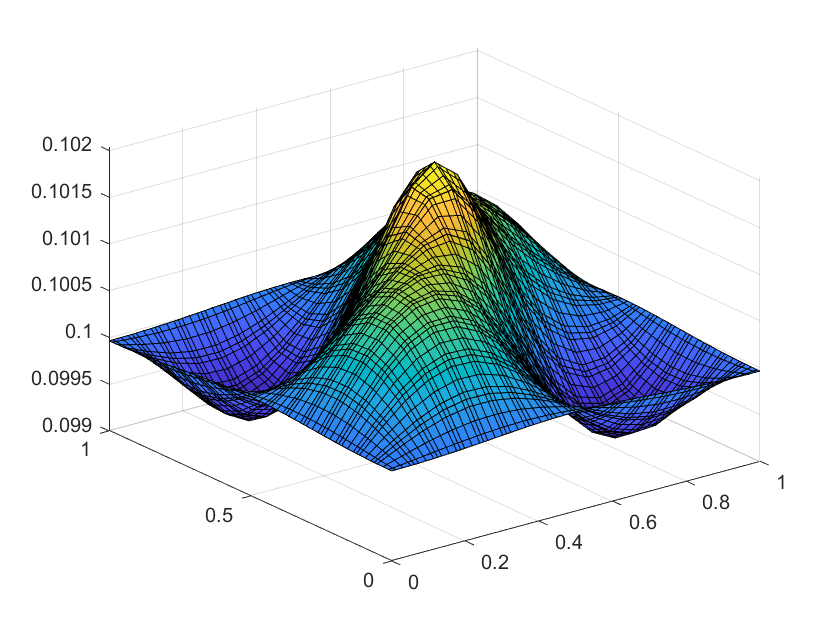}
		\end{minipage}
	}
    \hskip -0.3cm
	\subfigure[$c_h(x,y,0.05)$]{
		\begin{minipage}{0.33\textwidth}
			\includegraphics[width=1.0\textwidth]{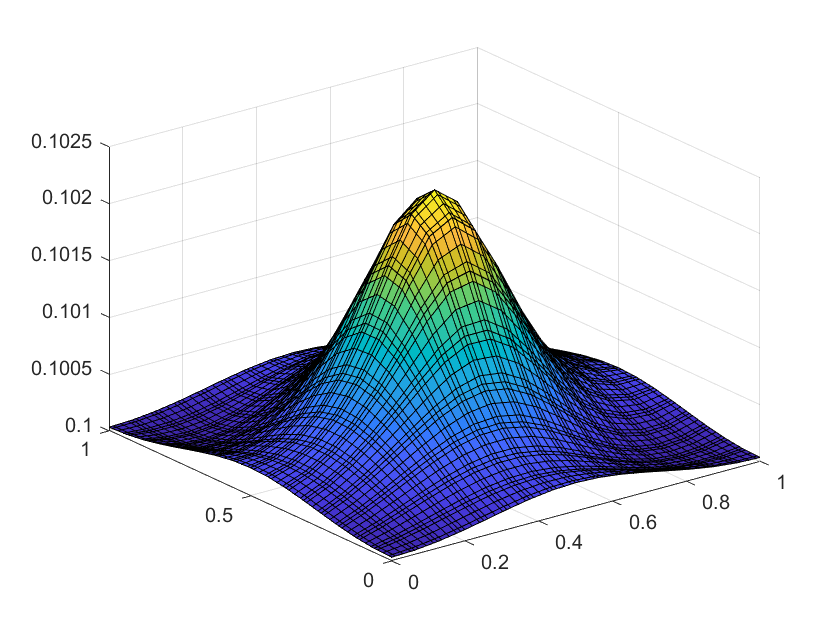}
		\end{minipage}
	}
    \hskip -0.3cm
	\subfigure[$c_h(x,y,0.1)$]{
		\begin{minipage}{0.33\textwidth}
			\includegraphics[width=1.0\textwidth]{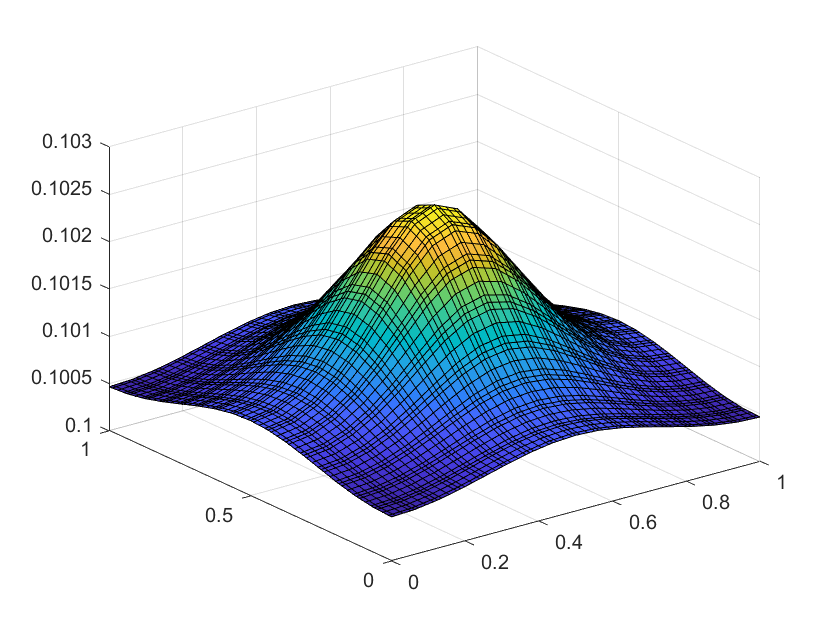}
		\end{minipage}
	}
    \hskip -0.3cm
	\subfigure[$c_h(x,y,0.5)$]{
		\begin{minipage}{0.33\textwidth}
			\includegraphics[width=1.0\textwidth]{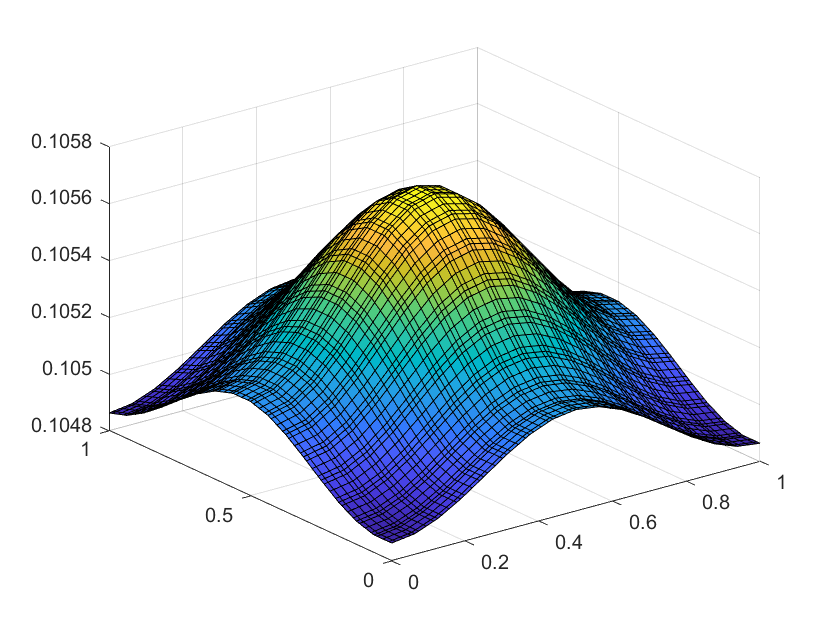}
		\end{minipage}
	} 
	\caption{The numerical solutions $u_h$ and $c_h$ at different time of Example \ref{eg3}.}
	\label{fig}
\end{figure}
For this example, the initial conditions are chosen as local Gaussian $u(x,y,0)$ and sinusoidal disturbance $c(x,y,0)$. 
The total mass of the cells is calculated to be 0.004$\pi$. This value is significantly small, especially in the context of the KS model, where the critical mass threshold for blow-up (e.g., 8$\pi$ in two dimensions) is much higher. Since 0.004$\pi$ is below far this critical threshold, it indicates that the system is unlikely to experience a blow-up. Instead, cells are expected to spread out or reach a stable equilibrium state over time due to diffusion and chemotaxis, rather than concentrating to a point where blow-up occurs.  
We set $k=1, N_x = 32$.
In Figure \ref{fig}, for a short time, the initial Gaussian distribution (peak in the center $(0.5,0.5)$) will gradually spread outward due to the diffusion term $D\nabla u$, forming a smooth decaying profile. The image of $c$ shows periodic ripples, the amplitude decreasing over time, but superimposing a smooth bump in the peak area of $u$. With the evolution of time, the image of $c$ gradually loses high-frequency oscillation and presents a unimodal distribution centered on $u$ (see Fig. 8.4(j)). For long periods $t=0.5$, the KS system tends to a spatially uniform steady state, the distribution of $u$ and $c$ tends to be flat, and the residual local structure completely disappears.
%{\color{red}{Check mass conservation
%whether remains constant over time (typical of KS systems)}}
\begin{example}\label{eg4}
   Two-dimensional blow up phenomena.
    
    In this example, we test the ability of the proposed method to capture the blow-up solution of the KS model.  The parameters are selected as $\chi=D=\beta=\alpha=1$, $\varphi(u)=u$, and domain $\Omega=[-0.5,0.5]^2$.
We use the Neumann boundary conditions, $\nabla u\cdot {\bf n}=\nabla c\cdot {\bf n}=0, (x,y)\in \partial \Omega$.
    \end{example}
    
The initial solutions are
    \begin{eqnarray*}
        u(x,y,0)=840\exp(-84(x^2+y^2)),\quad c(x,y,0)=420\exp(-42(x^2+y^2)).
    \end{eqnarray*}
    According to \cite{li2017,qiu2021}, the numerical solution explodes in a finite time at the center of a rectangular region.
Figure \ref{fig1blwup} presents the numerical solutions of the cell density $u(x, y, t)$ for Example 6.4 obtained by solving the two-dimensional Keller-Segel chemotaxis model using the method proposed in this paper at $t=0$, $t = 10^{-5}$ and $t = 5 \times 10^{-5}$. 
%The time step is $dt=T/10\times{Nx}^2 $, 
 We
  set $k=2, N_x = 32$. From these figures, it can be observed that, as expected, the example exhibits a blow-up phenomenon in the center of the rectangular domain $[-0.5, 0.5] \times [-0.5, 0.5]$ in finite time. As time increases, the maximum peak value of the cell density $u(x, y, t)$ forms a very sharp peak structure in the central region. Specifically, Figures 6.7(b) and (c) show the cell density $u(x, y, t)$ just before the blow-up at $t = 10^{-5}$ and  the blow-up phenomena at $t = 5 \times 10^{-5}$, respectively.
	\begin{figure}[!ht]
	\centering
	\subfigure[$u_h(x,y,0)$]{
		\begin{minipage}{0.33\textwidth}
			\includegraphics[width=1.0\textwidth]{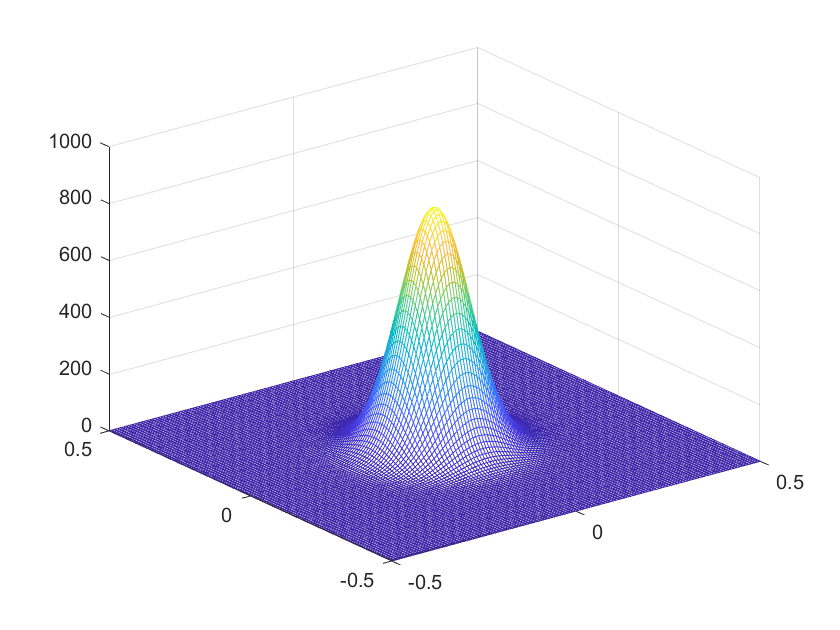}
		\end{minipage}
	}
    \hskip -0.3cm
	\subfigure[$u_h(x,y,10^{-5})$]{
		\begin{minipage}{0.33\textwidth}
			\includegraphics[width=1.0\textwidth]{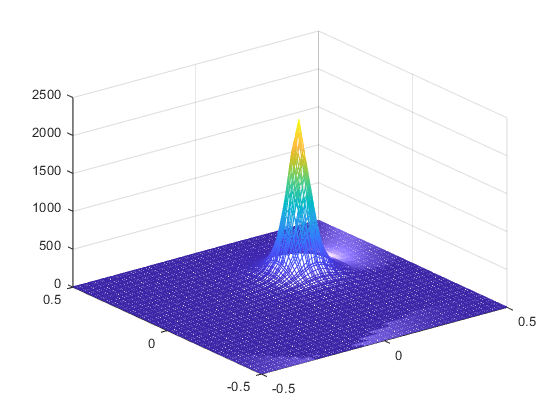}
		\end{minipage}
	}
    \hskip -0.3cm
	\subfigure[$u_h(x,y,5\times10^{-5})$]{
		\begin{minipage}{0.33\textwidth}
			\includegraphics[width=1.0\textwidth]{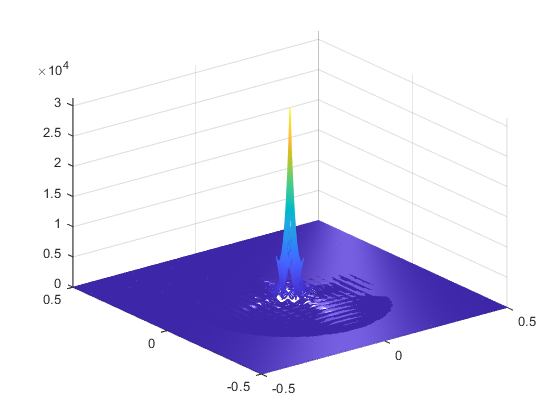}
		\end{minipage}
	}	
	
	\caption{ The blow up phenomenon of Example \ref{eg4}.}
	\label{fig1blwup}
\end{figure}

\section{Conclusions}\label{sec:conclusions}
In this work we devise and analyze a decoupled discrete scheme for KS system that combines an energy-stable semi-implicit method with a positivity-preserving DDG method. 
By means of a carefully designed energy functional, 
we prove that the cell density and the chemoattractant concentration remain strictly positive at every Gauss–Lobatto node, 
while the discrete energy is non-increasing. 
In addition, we establish optimal-order error estimates that quantify the approximation accuracy of the numerical solutions. 
Extensive experiments are presented to validate the theoretical results and to illustrate the scheme’s ability to capture both non-aggregation and saturation-limited aggregation phenomena.

The present work lays a solid theoretical and computational foundation for positivity-preserving, 
energy-stable discretizations of Keller–Segel-type systems. 
Furthermore, the current first-order scheme can be extended to second- or higher-order methods 
(e.g., BDF2, RK2, or exponential integrators) while maintaining energy stability and positivity.

\section*{Acknowledgement}
Y. Qin is partially supported by the National Natural Science Foundation of China (No. 12201369) 
and the Fundamental Research Program of Shanxi Province (No. 202303021211004). 
X. Yin is partially supported by the National Natural Science Foundation of China (No. 12416626).
\bibliographystyle{plain}
\normalem
\bibliography{refs}

\begin{thebibliography}{10}

\bibitem{Aizinger2000}
V.~Aizinger, C.~Dawson, B.~Cockburn, and P.~Castillo.
\newblock Local discontinuous {Galerkin} methods for contaminant transport.
\newblock {\em Adv. in Water Res.}, 24:73--87, 2000.

\bibitem{cao2017}
W.~Cao, H.~Liu, and Z.~Zhang.
\newblock Superconvergence of the direct discontinuous {Galerkin} method for
  convection-diffusion equations.
\newblock {\em Numer. Meth. Part. D. E}, 33, 2017.

\bibitem{cao2025}
W.~Cao, Y.~Qin, and M.~Xu.
\newblock A class of positive-preserving, energy stable and high order
  numerical schemes for the {Poission-Nernst-Planck} system.
\newblock {\em arXiv preprint}, arXiv:2502.03892, 2025.

\bibitem{che2018}
A.~Chenock, Y.~Epshteyn, and H.~Hu.
\newblock High-order positivity-preserving hybrid
  finite-volume-finite-difference methods for chemotaxis systems.
\newblock {\em Adv. Comput. Math.}, 44:327--350, 2018.

\bibitem{che2008}
A.~Chertock and A.~Kurganov.
\newblock A second-order positivity preserving central-upwind scheme for
  chemotaxis and haptotaxis models.
\newblock {\em Numer. Math. (Heidelb)}, 111:169--205, 2008.

\bibitem{chj1981}
S.~Chjldress and J.K. Percus.
\newblock Nonlinear aspects of chemotaxis.
\newblock {\em Math. Biosci.}, 8:217 -- 237, 1981.

\bibitem{cockburn1998}
B.~Cockburn and C.-W. Shu.
\newblock The local discontinuous {Galerkin} method for convection diffusion
  systems.
\newblock {\em SIAM J. Numer. Anal.}, 35:2440--2463, 1998.

\bibitem{corrias2014}
L.~Corrias, M.~Escobedo, and J.~Matos.
\newblock Existence, uniqueness and asymptotic behavior of the solutions to the
  fully parabolic {Keller-Segel} system in the plane.
\newblock {\em J. Differ. Equations}, 257(6):1840--1878, 2014.

\bibitem{corrias2004}
L.~Corrias, B.~Perthame, and H.~Zaag.
\newblock Global solutions of some chemotaxis and angiogenesis systems in high
  space dimensions.
\newblock {\em Milan J. Math}, 72:1--28, 2004.

\bibitem{fil2006}
F.~Filbet.
\newblock A finite volume scheme for the {Patlak-Keller-Segel} chemotaxis
  model.
\newblock {\em Numer. Math. (Heidelb)}, 104:457--488, 2006.

\bibitem{Girault2005}
V.~Girault, B.~Rivi\'{e}re, and M.F. Wheeler.
\newblock A discontinuous {Galerkin} method with non-overlapping domain
  decomposition for the {Stokes and Navier-Stokes} problems.
\newblock {\em Math. Comput.}, 74:53--84, 2005.

\bibitem{gajewski1998}
G.~Herbert and K.~Zacharias.
\newblock Global behaviour of a reaction-diffusion system modelling chemotaxis.
\newblock {\em Math. Nachr.}, 195:77--114, 1998.

\bibitem{hes2007}
J.S. Hesthaven and T.~Warburton.
\newblock {\em Nodal Discontinuous {G}alerkin Methods: Algorithms, Analysis,
  and Applications}.
\newblock Springer, New York, 2007.

\bibitem{huang2012}
Y.~Huang, H.~Liu, and N.~Yi.
\newblock Recovery of normal derivatives from the piecewise {L}2 projection.
\newblock {\em J. Comput. Phys.}, 231:1230--1243, 2012.

\bibitem{Kel1970}
E.F. Keller and L.A. Segel.
\newblock Initiation of slime mode aggregation viewed as an instability.
\newblock {\em J. Theor. Biol.}, 26(3):399--415, 1970.

\bibitem{Kel1971}
E.F. Keller and L.A. Segel.
\newblock Model for chemotaxias.
\newblock {\em J. Theor. Biol.}, 30(2):225--234, 1971.

\bibitem{li2017}
X.~Li, C.-W. Shu, and Y.~Yang.
\newblock Local discontinuous {Galerkin} method for the {Keller-Segel}
  chemotaxis model.
\newblock {\em SIAM J. Sci. Comput.}, 73(5):943--967, 2017.

\bibitem{HLiu2015}
H.~Liu.
\newblock Optimal error estimates of the direct discontinuous {Galerkin} method
  for convection–diffusion equations.
\newblock {\em Math. Comput.}, 84:2263--2295, 2015.

\bibitem{liu2010}
H.~Liu and J.Yan.
\newblock The direct discontinuous {G}alerkin {(DDG)} method for diffusion with
  interface corrections.
\newblock {\em Commun. Comput. Phys.}, 8(3):541 -- 564, 2010.

\bibitem{liu2016}
H.~Liu and Z.~Wang.
\newblock An entropy satisfying discontinuous {G}alerkin method for nonlinear
  {Fokker-Planck} equations.
\newblock {\em J. Sci. Comput.}, 68(3):1217--1240, 2016.

\bibitem{liu2017}
H.~Liu and Z.~Wang.
\newblock A free energy satisfying discontinuous {G}alerkin method for
  one-dimensional {Poisson-Nernst-Planck} systems.
\newblock {\em J. Comput. Phys.}, 328:413--437, 2017.

\bibitem{liu2009}
H.~Liu and J.~Yan.
\newblock The direct discontinuous {Galerkin (DDG)} methods for diffusion
  problems.
\newblock {\em SIAM J. Numer. Anal.}, 47:675--698, 2009.

\bibitem{nagai2018}
T.~Nagai and T.~Yamada.
\newblock Boundedness of solutions to a parabolic-elliptic {Keller-Segel}
  equation in {$R^2$} with critical mass.
\newblock {\em Adv. Nonlinear Stud.}, 18:337--360, 2018.

\bibitem{osaki2001}
K.~Osaki and A.~Yagi.
\newblock Finite dimensional attractor for one-dimensional {Keller-Segel}
  equations.
\newblock {\em Funkcial. Ekvac.}, 44:441--470, 2001.

\bibitem{Pat1953}
C.S. Patlak.
\newblock Random walk with persistence and external bias.
\newblock {\em Bull. Math. Biophys.}, 15(3):311--338, 1953.

\bibitem{qiu2021}
C.~Qiu, Q.~Liu, and J.~Yan.
\newblock Third order positivit-preserving direct discontinuous {Galerkin}
  method with interface correction for chemotaxis {Keller-Segel} equations.
\newblock {\em J. Comput. Phys.}, 433:110191, 2021.

\bibitem{ri2008}
B.~Rivi\`{e}re.
\newblock {\em Discontinuous Galerkin Methods for Solving Elliptic and
  Parabolic Equations: Theory and Implementation.}
\newblock SIAM, Philadelphia, 2008.

\bibitem{saito2007}
N.~Saito.
\newblock Conservative upwind finite-element method for a simplified
  {Keller-Segel} system modelling chemotaxis.
\newblock {\em IMA J. Numer. Anal.}, 27(2):332--365, 2007.

\bibitem{saito2005}
N.~Saito and T.~Suzuki.
\newblock Notes on finite difference schemes to a parabolic-elliptic system
  modelling chemotaxis.
\newblock {\em Appl. Math. Comput.}, 171(1):72--90, 2005.

\bibitem{sul2019}
M.~Sulman and T.~Nguyen.
\newblock A positivity preserving moving mesh finite element method for the
  {Keller-Segel} chemotaxis model.
\newblock {\em J. Sci. Comput.}, 80(6):1--18, 2019.

\bibitem{sun2005}
S.~Sun and M.F. Wheeler.
\newblock Symmetric and nonsymmetric discontinuous {Galerkin} methods for
  reactive transport in porous media.
\newblock {\em SIAM J. Numer. Anal.}, 43:195–219, 2005.

\bibitem{winkler2010}
M.~Winkler.
\newblock Aggregation vs. global diffusive behavior in the higher-dimensional
  {Keller-Segel} model.
\newblock {\em J. Differ. Equations}, 248:2889--2905, 2010.

\bibitem{winkler20101}
M.~Winkler.
\newblock Does a 'volume-filling effect' always prevent chemotactic collapse?
\newblock {\em Math. Methods Appl. Sci.}, 33:12--24, 2010.

\bibitem{winkler2013}
M.~Winkler.
\newblock Finite-time blow-up in the higher-dimensional parabolic-parabolic
  {Keller-Segel} system.
\newblock {\em J. Math. Pures Appl.}, 100:748--767, 2013.

\bibitem{yin2014}
P.~Yin, Y.~Huang, and H.~Liu.
\newblock An iterative discontinuous {G}alerkin method for solving the
  nonlinear {Poisson-Boltzmann} equation.
\newblock {\em Commun. Comput. Phys.}, 16:491--515, 2014.

\bibitem{ZS2010}
X.~Zhang and C.-W. Shu.
\newblock On maximum-principle-satisfying high order schemes for scalar
  conservation laws.
\newblock {\em J. Comput. Phys.}, 229(9):3091--3120, 2010.

\bibitem{zhou2017}
G.~Zhou and N.~Saito.
\newblock Finite volume methods for a {Keller-Segel} system: {Discrete} energy,
  error estimates and numerical blow-up analysis.
\newblock {\em Numer. Math. (Heidelb)}, 135(1):265--311, 2017.

\end{thebibliography}
\end{document}